\documentclass[11pt]{amsart}
\usepackage[margin=1in]{geometry}

\usepackage{algorithm}
\usepackage[noend]{algorithmic}

\usepackage{graphicx}
\usepackage{amsfonts,amsmath,amssymb}
\usepackage{bm}
\newtheorem{theorem}{Theorem}[section]
\newtheorem{proposition}[theorem]{Proposition}
\newtheorem{lemma}[theorem]{Lemma}
\newtheorem{corollary}[theorem]{Corollary}
\newtheorem{definition}[theorem]{Definition}

\newcommand{\pmt}{\mathrm{Pr}_{\text{MT}}}
\newcommand{\mybinom}[3][0.8]{\scalebox{#1}{$\dbinom{#2}{#3}$}}

\newcommand{\ns}{\text{NS}}
\newcommand{\gns}{\text{GNS}}
\newcommand{\sns}{\text{SNS}}

\newcommand{\bE}{\ensuremath{\mathbf{E}}}

\DeclareMathOperator{\poly}{poly}

\begin{document}

\title{The Moser-Tardos Framework with Partial Resampling}

\author{
{\sc David G.~Harris}$^{1}$
\and
{\sc Aravind Srinivasan}$^{2}$
}

\setcounter{footnote}{0}
\addtocounter{footnote}{1}
\footnotetext{Department of Computer Science, University of Maryland, 
College Park, MD 20742. 
Research supported in part by NSF Awards CNS-1010789 and CCF-1422569.
Email: \texttt{davidgharris29@gmail.com}.}
\addtocounter{footnote}{1}
\footnotetext{Department of Computer Science and
Institute for Advanced Computer Studies, University of Maryland, 
College Park, MD 20742. 
Research supported in part by NSF Awards CNS-1010789, CCF-1422569, and CCF-1749864, a gift from Google, Inc., and research awards from Adobe, Inc. and Amazon, Inc.
Email: \texttt{srin@cs.umd.edu}.}

\date{}
\maketitle

\sloppy
\pagestyle{plain}

\setcounter{page}{1}
\begin{abstract}The resampling algorithm of Moser \& Tardos is a powerful approach to develop constructive versions of the Lov\'{a}sz Local Lemma (LLL). We generalize this to \emph{partial} resampling: when a bad event holds, we resample an appropriately-random \emph{subset} of the variables that define this event, rather than the entire set as in Moser \& Tardos.  This is particularly useful when the bad events are determined by sums of random variables. This leads to several improved algorithmic applications in scheduling, graph transversals, packet routing etc. For instance, we settle a conjecture of Szab\'{o} \& Tardos (2006) on graph transversals asymptotically, and obtain improved approximation ratios for a packet routing problem of Leighton, Maggs, \& Rao (1994). 
\end{abstract}

\smallskip \noindent
\textbf{Conference versions of this work.} Preliminary versions of parts of this paper appeared in two papers by the
authors: \cite{harris-srin-assign-lll,harris-srin:focs13}. 

\section{Introduction}
\label{sec:intro}
The Lov\'{a}sz Local Lemma (LLL) \cite{erdos-lovasz:lll} is a fundamental probabilistic tool. The breakthrough of Moser \& Tardos \cite{moser-tardos:lll} gives a constructive approach to the LLL through a very natural resampling procedure, which we summarize as follows. Suppose we have a collection $\mathcal B$ of ``bad" events, each $B \in \mathcal B$ being a Boolean function of  a subset  of \emph{independent} random variables $X_1, X_2, \ldots, X_n$. Then, assuming that the standard sufficient conditions of the LLL hold, the following resampling algorithm (which we refer to as the \emph{MT algorithm}) quickly converges to a setting of the $X_j$'s that simultaneously avoids all the bad-events in $\mathcal B$:
\begin{itemize}
\item Sample $X_1, \dots, X_n$ independently from their respective distributions;
\item \textbf{while} some bad event is true, pick one of these, say $B$, arbitrarily, and resample (independently) all the variables determining it.
\end{itemize}
We generalize this to an algorithm that we call the Partial Resampling Algorithm (PRA); the idea is that when a bad-event $B$ is true, we randomly select 
a \emph{subset} $D$ of its variables, according to some carefully-designed probability distribution, and then only resample the variables $X_j$ contained in $D$. This partial-resampling approach leads to algorithmic results for many applications that are not captured by the LLL. 

As a starting point, suppose the bad-events are defined by non-negative linear threshold functions. In such cases, the constraints (i.e., the complements of the $B \in \mathcal B$) have the form
\begin{equation}
\label{eqn:packing-constraint}
 ~\sum_{i,j} a_{k,i,j} [X_i = j] \leq b_k.
\end{equation}

Here we use the Iverson notation: given an event $\mathcal{E}$, the notation $[\mathcal{E}]$ will stand for the indicator variable for $\mathcal{E}$. Thus $[X_i = j]$ is the indicator variable that variable $X_i$ takes on value $j$.  The matrix $A$ of coefficients $a_{k,i,j}$ has $m$ rows indexed by $k$, and $N$ columns that are indexed by pairs $(i,j)$;  by scaling, we will assume throughout that $a_{k,i,j} \in [0,1]$.

Low-congestion routing is a motivating example of this type of problem: we are given a collection of (source, destination) pairs $(s_1,t_1), \dots, (s_n, t_n)$ in a directed or undirected graph $G$ with edge-set $E$; each edge $f \in E$ has a capacity $b_f$, and we are also given a collection $\mathcal{P}_i = \{P_{i,j}\}$ of possible routing paths for each $(s_i,t_i)$-pair. We aim to choose one path from each collection $\mathcal{P}_i$, in order to minimize the \emph{relative congestion}: the minimal $T$ such that the maximum load on any edge $f$ is at most $T \cdot c_f$.  This problem leads to a simple IP formulation:
$$
\text{minimize } T ~\mbox{subject to}~ \left[\forall i, ~\sum_j z_{i,j} = 1; ~~
\forall f \in E, ~\sum_{(i,j):~f \in P_{i,j}} z_{i,j} \leq T \cdot c_f; ~~z_{i,j} \in \{0,1\}. \right]
$$
where $z_{i,j}$ is an indicator variable for selecting path $P_{i,j}$.

There are two broad approaches to such problems, both starting with the natural LP relaxation where the variable $z_{i,j} \in [0,1]$ represents the fractional assignment of variable $X_i$ to value $j$.  Suppose that $z$ satisfies the constraints $\sum_{i,j} a_{k,i,j} z_{i,j} \leq c_k$. The natural question is: 
\begin{quote}
``What choice of vector $b$, and what conditions on the matrix $A$ and vector $c$,  ensure that (\ref{eqn:packing-constraint}) has an integer solution, which, furthermore, can be found efficiently?" 
\end{quote}

The first major approach to this is polyhedral. Letting $D$ denote the maximum column sum of $A$, i.e., $D = \max_{i,j} \sum_k a_{k,i,j}$, the rounding theorem of \cite{klrtvv} shows constructively that for all $k$, 
\begin{equation}
\label{eqn:klrtvv}
b_k = c_k + D
\end{equation}
suffices; for the low-congestion routing problem, for instance, this would show that if each path has length $D$, then we can obtain an $O(D)$-approximate solution.

The second approach is randomized rounding \cite{raghavan-thompson}: independently set $X_i = j$ with probability  $z_{i,j}$. The standard ``Chernoff bound followed by a union bound over all $m$ rows" analysis of \cite{raghavan-thompson} shows that this works for
\begin{equation} 
\label{eqn:rand-round}
b_k =
\begin{cases}
C \cdot \frac{\log m}{1 + \log( \frac{\log m}{c_k} )} & \text{if $c_k \leq \log m$} \\
c_k + C \sqrt{c_k \log m} & \text{if $c_k > \log m$}
\end{cases}
\end{equation}
where $C$ is some universal constant. In particular, the low-congestion routing problem can be approximated to within $O( \frac{\log m}{ \log\log m} )$, where $m$ denotes the total number of edges. 

Let us compare these known bounds (\ref{eqn:klrtvv}) and (\ref{eqn:rand-round}). The former is good when all the $c_k$ are ``large" (say, much bigger than, or comparable to, $D$); the latter is better when $D$ is large compared to $m$. Can we do better? Our Theorem~\ref{csp-thm}  answers this in the affirmative -- we replace $m$ by $D$ in (\ref{eqn:rand-round}), showing constructively that when $c_k = R \geq 1$ we may set
\begin{equation}
\label{eqn:rand-round2}
b_k = 
\begin{cases}
C \frac{\log D}{1+\log(\frac{\log D}{R})} & \text{if $R \leq \log D$} \\
R + C \sqrt{R \log D} & \text{if $R \geq \log D$}
\end{cases}
\end{equation}

Thus, for the low-congestion routing problem, this would give an approximation ratio $O( \frac{\log D}{\log \log D})$, beating both the union bound and the polyhedral bounds.

We will show in Appendix~\ref{compare-mt-sec} that the MT algorithm cannot directly get ``scale-free'' bounds such as (\ref{eqn:rand-round2}) (that is, $b_k$ is a function of $R, D$ but not of $n$ or $m$). In such cases, MT algorithm is no better than random search, requiring exponential time.

There are two other related results which deserve mention here. First, \cite{llrs,harris-srin-assign-lll} shows a bound similar to (\ref{eqn:rand-round2}) but with $D'$, the maximum number of non-zeroes in any column of $A$, playing the role of $D$. Note that $D' \geq D$ always, and that $D' \gg D$ is possible. Moreover, the bound of \cite{llrs} primarily works when all the $c_k$ are close to each other, and rapidly degrades when these values can be disparate; the bound of \cite{harris-srin-assign-lll} is nonconstructive. 

While we have discussed linear threshold functions here for simplicity, much of the power of the PRA comes from the fact that it is flexible enough to handle complex constraints which have both linear and non-linear components and thus behave ``almost linearly.'' By contrast, approaches such as \cite{klrtvv}, which depend critically on linear algebra, and results such as \cite{llrs} based on multiple phases of resampling, are more difficult to adapt to such problems. Let us summarize two non-linear problems where we obtain improved bounds:

\noindent \textbf{Transversals with omitted subgraphs.} Given a partition of the vertices of an undirected graph $G$ into blocks, a \emph{transversal} is a subset of the vertices, one chosen from each block. 
An \emph{independent transversal}, or independent system of representatives, is a transversal that is also an independent set in $G$. The study of independent transversals was initiated by Bollob\'{a}s, Erd\H{o}s \& Szemer\'{e}di \cite{bollobas-erdos-szemeredi:isr}, and has received a considerable amount of attention (see, e.g., \cite{aharoni-berger-ziv,alon:lin-arb,alon:strong-chi,haxell:struct-indep-set,haxell-szabo-tardos:partitioning,jin:transversals,locally-sparse,szabo-tardos:extremal,yuster:transversals}). Such transversals serve as building blocks for other graph-theoretic parameters such as the linear arboricity and strong chromatic number \cite{alon:lin-arb,alon:strong-chi}.

We improve (algorithmically) a variety of sufficient conditions for the existence of certain transversals. In particular,  Szab\'{o} \& Tardos present a conjecture on the minimum block size to guarantee the existence of transversals that avoid $K_s$ \cite{szabo-tardos:extremal}; we show that this conjecture holds asymptotically for large $s$.  

\noindent \textbf{Packet routing with low latency.} Consider an undirected graph $G$ with $N$ packets, in which we need 
to route each packet $i$ from vertex $s_i$ to vertex $t_i$ along a 
\emph{given simple path} $P_i$. 
The constraints are that each edge can carry only one packet at a time,
and each edge traversal takes unit time for a packet; edges are
allowed to queue packets. A well-known scheduling problem considered in \cite{lmr} is to minimize the \emph{makespan} $T$ (the time by which all packets are delivered).

 Two natural lower-bounds on
$T$ are the \emph{congestion} $C$ (the maximum number of the $P_i$ that
contain any given edge of $G$) and the \emph{dilation} $D$ (the
length of the longest $P_i$); thus, $(C + D)/2$ is a universal lower-bound, and there exist families of instances with
$T \geq (1 + \Omega(1)) \cdot (C + D)$ \cite{rothvoss:cplusd}. A seminal result of \cite{lmr} is
that $T \leq O(C + D)$ for all input instances, using 
constant-sized queues at the edges; the asymptotic notation hides a rather large constant. Building on further improvements \cite{schei:thesis,peis-wiese}, our work \cite{harris-srin-assign-lll} developed a nonconstructive $7.26(C + D)$ and a constructive $8.84(C + D)$ bound; we improve these further to a 
constructive $6.73  (C + D)$. 

\medskip \smallskip

\noindent
\textbf{Informal discussion of the Partial Resampling Algorithm.} 
To understand the intuition behind the PRA, consider bad-events of the  form  $[X_{i_1} = j_1] + \dots + [X_{i_k} = j_k] \geq t$, where the expected value of $[X_{i_1} = j_1] + \dots + [X_{i_k} = j_k]$ is $\mu < t$. Suppose we run the MT algorithm on this problem. We begin by drawing all the variables $X_1, X_2, \ldots, X_n$ independently from their original distributions, and we find some such bad-event has become true. At this point, the MT algorithm would resample all of the variables affected by this event --- that is, all the variables $X_{i_1}, \dots, X_{i_k}$. 

But suppose now that $X_{i_\ell} \neq j_\ell$ for some $\ell \in [k]$. In that case, variable $X_{i_\ell}$ seems like it is ``helpful'' in terms of avoiding this bad-event. Since the goal of resampling a bad-event is to ``fix'' it, then resampling this $X_{i_\ell}$ seems counter-productive. Thus it seems more appropriate to only resample the variables with $X_{i_\ell} = j_\ell$. In the original step of the MT algorithm, the expected number of such variables is $\mu$, and we expect that in intermediate stages of the MT algorithm it should also be close to $\mu$. Thus, heuristically, we should only be resampling about $\mu$ variables, not all $k$ variables.
In fact, even this is somewhat too many variables to resample. Since we expect about $\mu$ variables to have $X_{i_\ell} = j_\ell$, it is only the $t - \mu$ ``extra'' variables which are causing the bad-event to occur. 
Thus, we only sample some of the $X_{i_j}$'s, and we make this choice probabilistically.

The power of the partial-resampling approach comes from the fact it makes steady progress toward a solution --- when there is a bad-event, we make a minimal change to fix it, while preserving as much of the prior solution as possible. For these linear-threshold bad-events, it is only these few, ``guilty'' variables which should be resampled. We thus make much smaller steps to fix any bad-event, making steady progress toward a solution which avoids them all.\footnote{There is an alternative way to apply the LLL in this context, which is to define a separate bad-event for each \emph{atomic} bad configuration, that is, for each subset of variables which exceed value of $t$. This subdivides the original bad-event into approximately $\binom{k}{t}$ separate smaller bad-events. This approach can be effective in some regimes, especially when $t \gg \mu$, and can lead to scale-free configurations. However, this method suffers from the drawback that the sum of the probabilities of these atomic bad-events is much larger than the original probability of the single bad-event (because these atomic events have significant positive correlation). The method we develop will be strictly stronger than this approach.}

In general, the PRA tends to work well when there are common configurations, which are not actually forbidden, but are nonetheless ``bad'' in the sense that they increase the probability of a bad-event. In the case of a sum of random variables, for example, this occurs whenever $X_{i_\ell} = j_\ell$ holds simultaneously for many values of $\ell$. We will see other examples of more complicated types of bad-but-legal configurations.

\smallskip \noindent \textbf{Organization of the paper.} The PRA is discussed in detail in Section~\ref{sec:resampling}. We give criteria, similar to the cluster-expansion LLL, asymmetric LLL, and symmetric LLL, for showing that this algorithm terminates in expected polynomial time. 
Section~\ref{complex-bad-events} describes how to apply the PRA when the bad-events are complex (such as linear-threshold functions) as opposed to pure atomic events. It also provides a more succinct formulation which allows us to reduce the number of parameters. 
Section~\ref{sec:random-sum} shows how to apply the PRA when the underlying bad-events are linear threshold functions. Such events are ubiquitous in combinatorics and algorithms, and the PRA deals with them particularly effectively.
We discuss applications to column-sparse packing (Section~\ref{sec:column-sparse}), transversals with omitted subgraphs (Section~\ref{sec:transversals}), and packet routing (Section~\ref{sec:routing}).

\section{The Partial Resampling Algorithm}
\label{sec:resampling}

\subsection{The variable-assignment setting}
Our algorithms and problems all come from a general class of constraint-satisfaction problems (CSP's) we refer to as the \emph{variable-assignment setting}. We have $n$ variables $X_1, \dots, X_n$; each variable has a finite set $F_i$ of possible values. We define the probability space $\Omega$ in which the variables are assigned independently: namely for each $i \in [n]$ we set $X_i = j$ with probability $p_{i,j}$, where  $j \in F_i$ ranges over the set of valid assignment to variable $i$.  We often omit the set $F_i$ when it is clear from context, e.g., we write simply $\sum_j p_{i,j} = 1$.

As a starting point for our algorithm, we assume there is a  collection $\mathcal B$ of bad-events to avoid, which are all \emph{atomic events} in that each bad-event $B$ can be written as a conjunction of the form
$$
B \equiv \Bigl( (X_{i_1} = j_1) \wedge \dots \wedge (X_{i_k} = j_k) \Bigr)
$$
for $k \geq 1$.

We refer to any ordered pair $( i, j)$ where $j \in F_i$, as an \emph{element}, and let $\mathcal X$ denote the set of all elements. We may encode these types of atomic bad-events by subsets of $\mathcal X$. We thus define the \emph{family of atomic sets} $\mathcal A \subseteq 2^{\mathcal X}$ to be the collection of all non-empty sets $Y \subseteq \mathcal X$ of the form
$$
Y = \{ (i_1, j_1), \dots, (i_k, j_k) \} \qquad \text{where $i_1, \dots, i_k$ are distinct and $j_{\ell} \in F_{i_{\ell}}$ for $\ell = 1, \dots, k$}
$$

We say that an atomic set $Y \in \mathcal A$ \emph{holds} on a configuration $X_1, \dots, X_n$, if $X_i = j$ for all $(i,j) \in Y$. Thus, an atomic event $B \equiv ((X_{i_1} = j_1) \wedge \dots \wedge (X_{i_k} = j_k))$ corresponds to the atomic set $A = \{ (i_1, j_1), \dots, (i_k, j_k) \}$. Here, $B$ is true on $X$ iff $A$ holds on $X$. So $\mathcal B$ can be regarded as subset of $\mathcal A$.

Given any vector $\lambda \in [0,1]^{\mathcal X}$, and any $Y \in \mathcal A$, we define 
\begin{equation}
\label{eqn:power-notation}
\lambda^{Y} = \prod_{x \in Y} \lambda_x
\end{equation}

Thus, for instance, for any $Y \in \mathcal A$, we have $P_{\Omega} (\text{$Y$ holds on $X$}) = p^Y$.

\subsection{Partial Resampling Algorithm and the Main Theorem}
\label{pra-first}
Our algorithm is driven by a parameter, the \emph{fractional hitting-set}, which tells which variables to resample for which bad-event.
\begin{definition}[\textbf{Fractional hitting-set}]
\label{defn:C}
Let $C: \mathcal A \rightarrow [0,1]$ and let $B \in \mathcal A$. We say that $C$ is a \emph{fractional hitting-set for $B$} if
$$
\sum_{\substack{Y \subseteq B \\ Y \neq \emptyset}} C(Y) \geq 1;
$$
 note that this summation treats $Y$ and $B$ as subsets of $\mathcal X$. 

We say that $C$ is a \emph{fractional hitting-set for $\mathcal B$} if $C$ is a fractional hitting-set for every $B \in \mathcal B$.
\end{definition}

We now introduce our main algorithm, which we refer to as the \emph{Partial Resampling Algorithm (PRA)}; this requires as input a fractional hitting-set $Q$ for the set of bad-events $\mathcal B$.
\begin{algorithm}[H]
\caption{$\text{PRA}(p, \mathcal B, Q, \mathcal X) $}
\begin{algorithmic}[1]
\STATE Draw the variables $X_1, \dots, X_n$ independently, where $X_i = j$ with probability $p_{ij}$.
\WHILE{there is a true bad-event $B \in \mathcal B$}
\STATE Select, arbitrarily, some true bad-event $B \in \mathcal B$. We refer to this set $B$ as the \emph{violated set}.
\STATE Select exactly one subset $Y \subseteq B$. The probability of selecting a given $Y$ is given by
$$
\Pr(\text{select $Y$}) = \frac{ Q
(Y) }{\sum_{Z \subseteq B} Q(Z)}.
$$
We refer to $Y$ as the \emph{resampled set}.
\STATE For each $(i,j) \in Y$, draw a new value for $X_i$ independently, using the probability distribution $p_i$. (We refer to this as \emph{resampling} $Y$).
\ENDWHILE
\RETURN $X$
\end{algorithmic} 
\end{algorithm}

The main difference between the PRA and the MT algorithm is that the latter would resample \emph{all} the variables of a true bad-event. The PRA only resamples a (carefully-chosen, random) subset of these variable.  In fact, with an appropriate choice of fractional hitting-set, the \emph{trivial hitting-set}, the PRA essentially degenerates into the MT algorithm.

\begin{definition}[\textbf{Trivial hitting-set}]
We define the \emph{trivial hitting-set} for $\mathcal B$ by  $C(Y) = [Y \in \mathcal B]$ (using Iverson notation as before).
\end{definition}

We will need to keep track of how bad-events (and subsets of bad-events) can be interdependent. This is more complicated than the usual LLL setting, because we distinguish two different types of dependencies: sets $Y, Y'$ could share a variable, \emph{or} they could both be potential resampling targets for some bad-event. The LLL only needs to keep track of the first type of dependency. We introduce the symmetric relations $\sim, \bowtie, \approx$ over $\mathcal A$ to account for these dependency types. These definitions all depend on some fixed choice of fractional hitting-set $Q$; we omit this from the notation for readability.

\begin{definition}[\textbf{Supported event}]
We say that a set $Y \in \mathcal A$ is \emph{supported} if $Q(Y) > 0$.
\end{definition}

\begin{definition}
[\textbf{Symmetric relations $\sim$, $\bowtie$, and $\approx$ on $\mathcal A$}]
\label{defn:symm-rels}
Let $Y, Y' \in \mathcal A$.  We say $Y \sim Y'$ iff
there exists a triple $(i,j,j')$ such that $(i, j) \in Y$ 
and $(i, j') \in Y'$: i.e., iff $Y$ and $Y'$ overlap in a 
variable. We also write $i \sim Y$ (or $Y \sim i$) to mean that $(i, j) \in Y$ for some $j$.

We say $Y \bowtie Y'$ iff $Y \not \sim Y'$ and
there is some event $B \in {\mathcal B}$ with $Y, Y' \subseteq B$. 

We say $Y \approx Y'$ iff either (1) $Y \sim Y'$ or (2) $Y \bowtie Y'$.

We say that $\bowtie$ is \emph{null} if for all supported $Y, Y'$ we have $Y \not \bowtie Y'$.
\end{definition}

Note that if $Y, Y' \subseteq B$ for some bad-event $B$, then $Y \approx Y'$.

In Theorem~\ref{resample-main-thm}, we give three conditions for the PRA to terminate. These conditions are analogous to, respectively, the cluster-expansion LLL criterion \cite{bissacot}, the asymmetric LLL, and the symmetric LLL. Conditions (b) and (c) follow immediately from (a), however we include (b) and (c) here since they are more convenient for typical applications. We begin with a preliminary definition.
\begin{definition}[\textbf{Neighbor-set for $Y$}]
For any  $Y \in \mathcal A$, we say that a set $\mathcal T \subseteq  \mathcal A$ is a \emph{neighbor-set} for $Y$ (and we write $\mathcal T \in \ns(Y)$) if the following conditions hold:
\begin{enumerate}
\item Every $Z \in \mathcal T$ satisfies $Z \approx Y$.
\item There do not exist distinct $Z, Z' \in \mathcal T$ with $Z \sim Z'$.
\item There is \emph{at most} one $Z \in \mathcal T$ with $Z \bowtie Y$.
\end{enumerate}
\end{definition}

\begin{theorem}[\textbf{Main Theorem}]
\label{resample-main-thm}
Let $Q$ be a fractional hitting-set for $\mathcal B$, and let $p \in [0,1 ]^{\mathcal X}$ be the probability vector for $\mathcal X$. In each of the following three cases, the PRA terminates in a feasible configuration avoiding all bad events with probability one. 

\smallskip \noindent 
(a) Suppose that $\mu: \mathcal A \rightarrow [0, \infty)$ satisfies, for all $Y \in \mathcal A$, 
$$
\mu(Y) \geq p^Y Q(Y) \sum_{\mathcal T \in \ns(Y)} \prod_{Y' \in \mathcal T} \mu(Y')
$$
Then, the expected number of resamplings is at most $\sum_{Y} \mu(Y)$.

\smallskip \noindent 
(b) Suppose that $\mu: \mathcal A \rightarrow [0, \infty)$ satisfies, for all $Y \in \mathcal A$, 
$$
\mu(Y) \geq p^Y Q(Y) \Bigl( \prod_{Y' \sim Y} (1 + \mu(Y')) \Bigr) \Bigl( 1 + \sum_{Y'' \bowtie Y} \mu(Y'')  \Bigr)
$$
Then, the expected number of resamplings is at most $\sum_{Y} \mu(Y)$.

\smallskip \noindent 
(c) Suppose that  $p^Y Q(Y) \leq P$ for all $Y \in \mathcal A$; and suppose that for all supported $Y$, there are \emph{at most} $D$ supported $Y'$ with $Y' \approx Y$ (we allow $Y' = Y$ here). And suppose finally that $e P D \leq 1$. Then, the expected number of resamplings is at most $e \sum_{Y} p^Y Q(Y)$.
\end{theorem}

The proof of Theorem~\ref{resample-main-thm} is lengthy; we will spend the next sections showing a number of preliminary lemmas.

\subsection{Analyzing the PRA: witness trees and the resampling table}
Our analysis of the PRA is based on \emph{witness trees} and a coupling construction called the \emph{resampling table}, two proof techniques developed by Moser \& Tardos in \cite{moser-tardos:lll}. 

We first describe the resampling table. In the PRA we have described, we generate new values for the variables $X_i$ as they are needed. Instead, we may imagine generating a resampling table $R(i,k)$ where $i$ ranges over $[n]$ and $k$ ranges over positive integers, whose entries are independent random variables where entry $R(i,k)$ is drawn from the distribution $p_i$. Once we have drawn this table, we use its values in place of resampling. For instance, in step (1) of the PRA, we set $X_i = R(i,1)$ for every $i \in [n]$. The first time we resample $X_i$, we set $X_i = R(i,2)$, and so on. It is clear that pre-generating the randomness does not affect the behavior of the PRA. We also observe that after drawing the resampling table, the only source of randomness remaining in the PRA is the choice of which subset $Y \subseteq B$ to select when resampling a bad-event $B$.

The witness tree $\hat \tau^t$ describes the history of the variables involved in the resampling up to some time $t$. Suppose we run the PRA for $t$ timesteps (not necessarily to completion) and that at each time $k = 1, \dots, t$ the violated bad-event is $B_k$, with resampled set $Y_k \subseteq B_k$.  We define $U_k$ to the ordered pair $(Y_k, B_k)$ for each time $k$.  We will construct an associated random variable $\hat \tau^t$, which is a type of rooted labeled tree.

To begin, we recursively define $\hat \tau_k^t$ for $k= t, t-1, \dots, 1$. We begin with  $\hat \tau^t_t$ which has a single root node labeled $Y_t$.  Next, letting $L(v)$ denote the label of a node $v$, we define $M_k^t$ for each $k < t$ to be the set of nodes $v \in \hat \tau^t_{k+1}$ with either of the two properties:
\begin{enumerate}
\item $L(v) \sim Y_k$; OR
\item $L(v) \bowtie Y_k$ and $v$ does not have any child $u$ such that $L(u) \bowtie L(v)$,
\end{enumerate}

If $M_k^t$ is empty, then we update $\hat \tau^t_{k} = \hat \tau^t_{k+1}$. Otherwise, we select some node $v \in M_k^t$ of greatest depth in $\hat \tau_{k+1}^t$ (breaking ties arbitrarily), and form $\hat \tau^t_k$ by adding a new node which is a child of $v$ and is labeled by $Y_k$. We finish by defining $\hat \tau^t = \hat \tau^t_1$.  

Note that the witness tree $\hat \tau^t$ does \emph{not} record information about the violated sets $B_k$ themselves; this is a critical step to cut down the number of possible trees.

By convention, for $s > t$ we define $\hat \tau_s^t = \emptyset$ (the null tree, which does not contain any nodes -- not even the root node). If the PRA has terminated already by time $t$, then we also define $\hat \tau^t = \emptyset$.

We define a \emph{tree-structure} to be any rooted tree whose nodes are labeled from the set $\mathcal A$. Every $\hat \tau^t$ is a tree-structure. We say that a tree-structure $\tau$ \emph{appears} if $\hat \tau^t = \tau$ for any value of $t$. We define the \emph{weight} of tree-structure $\tau$ by
$$
w(\tau) = \prod_{v \in \tau} p^{L(v)} Q(L(v))
$$
\smallskip \noindent \textbf{Remark.} Recall the notation (\ref{eqn:power-notation}) in parsing the value ``$p^{L(v)}$" above.

We note a few simple facts about witness trees.
\begin{proposition}
\label{distinct-ts-prop}
If $t < t'$ and the PRA has not terminated by time $t'$, then $\hat \tau^t \neq \hat \tau^{t'}$.
\end{proposition}
\begin{proof}
Let $Y_t, Y_{t'}$ be resampled sets at times $t, t'$ respectively. If $Y_t \neq Y_{t'}$, then the result holds because the root nodes of $\hat \tau^t, \hat \tau^{t'}$ have different labels. Otherwise, note that for every time $s$ with $Y_s = Y_t$, we will have $M_s^t \neq \emptyset$ and $M_s^{t'} \neq \emptyset$ (since the root nodes of $\hat \tau^t$ and $\hat \tau^{t'}$ respectively will be in these sets). Thus, $\hat \tau^{t'}$ will contain strictly more nodes with label $Y_t$ than does $\hat \tau^t$.
\end{proof}

\begin{proposition}
\label{ind-layer-prop}
If $v, v'$ are  distinct nodes at depth $h$ in $\hat \tau^t$, then $L(v) \not \sim L(v')$. 
Furthermore, if both $v, v'$ are leaf nodes, then $L(v) \not \bowtie L(v')$.
\end{proposition}
\begin{proof}
Suppose that $v, v'$ correspond to resamplings at time $s, s'$ respectively where $s < s' \leq t$. If $L(v) \sim L(v')$, then  $v' \in \hat \tau^t_s$ and so $v' \in M_s^t$. So $v$ will be eligible to be placed as a child of $v'$ or some other node of greater depth. In particular, $v$ will be placed at depth strictly larger than $h$. If $v, v'$ are leaf nodes and $L(v) \bowtie L(v')$, then, again $v' \in M^t_s$ so the same argument would apply.
\end{proof}

\subsection{The Witness Tree Lemma.}
The key to analyzing the PRA is the following lemma:

\begin{lemma}[\textbf{Witness Tree Lemma}]
\label{couple-lemma}
Any tree-structure $\tau$ appears with probability at most $w(\tau)$.
\end{lemma}

We prove this in two stages. We first connect the behavior of a witness trees to the resampling table, and show certain necessary conditions on the entries of $R$. This part of the proof is nearly identical to that of Moser \& Tardos. Second, we show that even after fixing $R$, there is still enough randomness remaining in the PRA to further bound the probability of a tree-structure appearing. 

Let $r$ be a possible value for the resampling table. We say a tree-structure $\tau$ is \emph{compatible} with $r$ if there is a non-zero probability of $\hat \tau^t$ appearing, conditioned on $R = r$.

\begin{proposition}[\cite{moser-tardos:lll}]
\label{compat-prop}
The probability that tree-structure $\tau$ is compatible with $R$ is at most $\prod_{v \in \tau} p^{L(v)}$.
\end{proposition}
\begin{proof}
Let $i \in [n]$, and consider the set of nodes $v \in \tau$ such that $L(v) \sim i$. By Proposition~\ref{ind-layer-prop}, all such nodes are at distinct depths in $\tau$. Let us sort these nodes in decreasing order of depth as $v_1, \dots, v_{s_i}$; for $k = 1, \dots, s_i$ let us say that $(i,a_{ik}) \in L(v_k)$. 

Suppose that $\hat \tau^t = \tau$ for some time $t$. We claim that $R(i,1) = a_{i1}$. For, suppose not; consider the first time $t'$ that $L(v_1)$ is the resampled set (this must occur, since $v_1 \in \hat \tau^t$). At time $t'$, we have $X_i = a_1$; so, $X_1$ must have been resampled earlier, at time $t''$. Let $Y''$ be the resampled set at time $t''$; so $i \sim Y''$. But then when constructing $\hat \tau^t$, we would place an additional node labeled $Y''$ at greater depth than $v_1$, leading to a contradiction. 

Continuing in this way, we see that $R(i,k) = a_{ik}$ for $i \in [n]$ and $k = 1, \dots, s_i$. Since the entries of $R$ are all independent, the overall probability of this event is 
$$
\prod_{i=1}^n \prod_{k=1}^{s_i} p_{i, a_{ik}}
$$

If $(i,j) \in L(v)$ for any node $v \in \tau$, then this contributes exactly one factor of $p_{ij}$ to this expression. So, we can rearrange this term as 
\[
\prod_{i=1}^n \prod_{k=1}^{s_i} p_{i, a_{ik}} = \prod_{v \in \tau} \prod_{(i,j) \in L(v)} p_{ij} = \prod_{v \in \tau} p^{L(v)} \qedhere
\]
\end{proof}

Before the second part of proof of the Witness Tree Lemma, let us give some intuition to the role of $\bowtie$. 

Consider a tree-structure $\tau$ consisting of a singleton node labeled $Y$. We want to show that $\tau$ appears with probability at most $w(\tau) = p^Y Q(Y)$. We have already seen that the probability that $R$ is compatible with $\tau$ is at most $p^Y$. We would next like to say that the probability that $Y$ was selected as the resampled set is at most $Q(Y)$, giving us our probabilistic bound.

\emph{On any given instance in which $Y$ is eligible to be the resampled set}, the probability of selecting $Y$ is indeed at most $Q(Y)$. However, there may have been a long sequence of bad-events in which $Y$ was eligible yet not selected. With enough of these opportunities,  the probability of eventually selecting $Y$ approaches $1$. Thus in order to obtain a useful bound on the probability of selecting $Y$, one must distinguish in advance a \emph{specific} instance in which $Y$ is eligible.

Now observe that if $Y \subseteq B$ for some true bad-event $B$, but we instead select some other $Y' \subseteq B$ as the resampled set, then $Y' \approx Y$.  So $Y'$ would be added as a child of $Y$ in the witness tree. As $\tau$ is a singleton node, then a necessary condition for $\tau$ to appear is that \emph{$Y$ is selected the first time it is eligible to be selected.}  With some thought, we see that this event has probability at most $Q(Y)$.

The reader should bear this intuition in mind for the remainder of the proof. Proposition~\ref{tree-change-prop} extends this to larger tree-structures, which can have more complex interactions.

\begin{proposition}
\label{tree-change-prop}
Suppose the PRA has not terminated by time $s$, and let $U_s = (Y_s, B_s)$. For some integer $t \geq s$, define $J$ to be the set of leaf nodes $v$ of $\hat \tau_s^t$ such that $L(v) \subseteq B_s$. Then:
\begin{enumerate}
\item $J$ cannot contain two nodes at the same depth.
\item If $J \neq \emptyset$ and $v$ is the (unique) vertex in $J$ of greatest depth, then $L(v) = Y_s$ and $\hat \tau_{s+1}^t = \hat \tau_s^t - v$. 
\item If $J = \emptyset$, then $\hat \tau_{s+1}^t = \hat \tau_s^t$.
\end{enumerate}
\end{proposition}
\begin{proof}
If $s = t$, then $\hat \tau_s^t$ contains a single node $v$ labeled $Y_s$ and $\hat \tau_{s+1}^t = \emptyset$ and $J = \{ v \}$.  In this case all the three parts hold easily. So let assume that $s < t$.

Part (1) follows immediately from Proposition~\ref{ind-layer-prop}, noting that $J$ contains only  leaf nodes.

For (2), let $v$ be the greatest-depth node of $J$. Suppose that $L(v) = Y' \neq Y_s$. We must have $v \in \hat \tau_{s+1}^t$ (the only node added to $\hat \tau_s^t$ has label $Y_s$). As $Y_s \approx Y'$, we have $v \in M_s^t$, and so $\hat \tau_s^t$ would have a new leaf node $w$ labeled $Y_s$ at greater depth than $v$; but then $w \in J$ has greater depth than $v$, a contradiction. 

Thus, we have shown that $L(v) = Y_s$. Observe that $\hat \tau_s^t$ is either equal to $\hat \tau_{s+1}^t$, or has a single leaf node $w$ added to $\hat \tau_{s+1}^t$ labeled $Y_s$. In the latter case, $w$ will have label $Y_s$ and will be at greatest depth in $\hat \tau_s^t$, and so $w = v$ and part (2) follows. In the former case, $v \in \hat \tau_{s+1}^t$. Since $L(v) = Y_s$, we will have $v \in M_s^t$, so again $\hat \tau_s^t$ would have a new leaf node $w$ labeled $Y_s$ at greater depth than $v$; but then $w \in J$ has greater depth than $v$, a contradiction.

To show (3), suppose that $\hat \tau_s^t \neq \hat \tau_{s+1}^t$. Then $\hat \tau_s^t$ has a new leaf node $v$ labeled $Y_s$. This node $v$ would be in $J$, contradicting that $J = \emptyset$.
\end{proof}

We are now ready to prove the Witness Tree Lemma.
\begin{proof}[Proof of Lemma~\ref{couple-lemma}]
To simplify the notation, let us define $Q(v) := Q(L(v))$ for any node $v$ of $\tau$. In analyzing the PRA, there are two types of random variables: first, there is the resampling table $R$; second, there are the random variables $U_k$. By our coupling construction, the full table $R$ is constructed at time $0$; thus, we may view the PRA as a stochastic process which sequentially generates the  random variables $R, U_1, U_2, \dots, $.

We will prove a stronger result: for all pairs of integers $s, T$ with $1 \leq s \leq T$, and all tree-structures $\tau$, we have
\begin{equation}
\label{induct-prop}
\Pr \bigl( \bigvee_{t = s}^{T-1} \hat \tau_s^t = \tau \mid R, U_1, \dots, U_{s-1} \bigr) \leq \prod_{v \in \tau} Q(v)
\end{equation}

We may assume that $\tau \neq \emptyset$, as otherwise the RHS of (\ref{induct-prop}) equal one and the bound holds vacuously. We prove this by induction backward on $s$. The base case is $s = T$; in this case LHS of (\ref{induct-prop}) is zero so the bound holds vacuously. We move to the induction step. Suppose that we condition on $R, U_1, \dots, U_{s-1}$. If there are no more true bad-events, then (as $\tau \neq \emptyset$) the bound holds since the LHS of (\ref{induct-prop}) equals zero.  Otherwise, let us condition as well on the violated set $B_s$. Define $J$ to be the set of leaf nodes $v$ of $\tau$ such that $L(v) \subseteq B_s$.

By Proposition~\ref{tree-change-prop}, in order to $\tau$ to appear, it must either be the case that $J = \emptyset$, or that $J$ contains a unique node at greatest depth.

In the first case, then by Proposition~\ref{tree-change-prop} we have $\hat \tau_{s+1}^t = \hat \tau_s^t$. So a necessary condition to have $\tau_s^t = \tau$ is to have $\tau_{s+1}^t = \tau$ for some $t$ in the range $s+1 \leq t < T$. By the induction hypothesis, this has probability at most $\prod_{v \in \tau} Q(v)$ (even after we condition on the random variable $U_s$).

In the second case, suppose that the greatest-depth node of $J$ is $v$, with $L(v) = Y$.  Then Proposition~\ref{tree-change-prop} shows that $Y$ must be the resampled set at time $s$ and $\hat \tau^t_{s+1} = \hat \tau^t_s - v$.  The probability of selecting $Y$ is at most $\frac{Q(Y)}{\sum_{Z \subseteq B} Q(Z)} \leq Q(Y) = Q(v)$. 

By the induction hypothesis, the probability that there is some $t$ with $\hat \tau^t_{s+1}= \tau - v$ is at most $\prod_{u \in \tau - v} Q(u)$, even after conditioning on $R, U_1, \dots, U_s$. Because of this conditional probability bound, we can multiply the two probabilities; so the overall probability of both selecting $Y$ and having $\hat \tau_{s+1}^t = \tau - v$, is at most $Q(v) \times \prod_{u \in \tau - v} Q(u) = \prod_{u \in \tau} Q(u)$, and the induction again holds.

We now move on to prove the full result. By Proposition~\ref{compat-prop}, the resampling table is compatible with $\tau$ with probability at most $\prod_{v \in \tau} p^{L(v)}$. We have shown that 
$$
\Pr(\bigvee_{t = 1}^T \hat \tau^t = \tau \mid R) \leq \prod_{v \in \tau} Q(v)
$$
So, multiplying these two probabilities gives
$$
\Pr( \bigvee_{t = 1}^T \hat \tau^t = \tau) \leq \bigl( \prod_{v \in \tau} Q(v) \bigr) \bigl( \prod_{v \in \tau} p^{L(v)} \bigr) = w(\tau)
$$

Taking the limit as $T \rightarrow \infty$:
\[
\Pr( \bigvee_{t \geq 1} \hat \tau^t = \tau) = \lim_{T \rightarrow \infty} \Pr( \bigvee_{t=1}^T \hat \tau^t = \tau) \leq \lim_{T \rightarrow \infty} w(\tau) = w(\tau) \qedhere
\]
\end{proof}

\subsection{Finishing the proof: counting witness trees}
To finish the proof of Theorem~\ref{resample-main-thm} and show the convergence of the PRA, we must count the weight of certain classes of witness trees.

\begin{definition}[\textbf{Proper tree-structure}]
A tree-structure $\tau$ is \emph{proper} if it has the following property: for every node $v$ with children $w_1, \dots, w_s$, the labels $L(w_1), \dots, L(w_s)$ are all distinct and $\{ L(w_1), \dots, L(w_s) \} \in \ns(L(v))$.

We let $\Gamma$ denote the set of all proper tree-structures, and $\Gamma(Y)$ the set of non-empty proper tree-structures whose root node has label $Y$.
\end{definition}

It is clear that if $\tau$ is a proper tree-structure, all of its subtrees are proper tree-structures as well.  
\begin{proposition}
\label{proper-ts}
The tree-structure $\hat \tau^t$ is proper.
\end{proposition}
\begin{proof}
Consider some node $v \in \hat \tau^t$ with children $w_1, \dots, w_s$. By Proposition~\ref{ind-layer-prop}, we have $L(w_i) \not \sim L(w_j)$ for $i \neq j$. Hence, $L(w_1), \dots, L(w_s)$ are distinct. Let $\mathcal T = \{L(w_1), \dots, L(w_s) \}$; we have thus shown that $\mathcal T$ satisfies property (2) of the definition of neighbor-set.

To show part (1), suppose that $w_i$ is added as a child of $v$ in forming $\hat \tau^t_k$. So $v \in M_k^t$ and $L(w_i)$ is the resampled set $Y_j$ at time $k$. By definition of $M_k^t$, we have $Y_k \approx L(v)$ and so $L(w_i) \approx L(v)$.

To show part (3), suppose that $L(w_i) \bowtie L(v)$ and $L(w_j) \bowtie L(v)$, and suppose that $w_i, w_j$ correspond to resamplings at times $s, s'$ respectively where $s < s'$. But then note that $v$ already has a child $w_j$ in $\hat \tau^t_{s}$ with $L(w_j) \bowtie L(v)$, so that $v \notin M_{s}^t$. So $w_i$ cannot be added as a child at time $s$, a contradiction.
\end{proof}
 
\begin{proposition}
\label{wt-bound-prop}
If the function $\mu$ satisfies Theorem~\ref{resample-main-thm}(a), then every $Y \in \mathcal A$ has
$$
\sum_{\tau \in \Gamma(Y)} w(\tau) \leq \mu(Y)
$$
\end{proposition}
\begin{proof}
For any $Y \in \mathcal A$, let $T_h(Y)$ be the total weight of all $\tau \in \Gamma(Y)$ whose height is at most $h$. We show that $T_h(Y) \leq \mu(Y)$ for all $Y \in \mathcal A$ and $h \geq 0$, by induction on $h$. When $h = 0$ this is vacuously true. Now, consider some $\tau \in \Gamma(Y)$ of height at most $h$. Suppose that the children of the root node are $w_1, \dots, w_s$ (with possibly $s = 0$), with labels $Y_1, \dots, Y_s$ respectively.

If $\tau_1, \dots, \tau_s$ are the subtrees of each $w_i$, then each $\tau_i$ has height at most $h-1$ and $\tau_i \in \Gamma(Y_i)$.  Also, note that $w(\tau) = p^Y Q(Y) w(\tau_1) \cdots w(\tau_s)$. So, for a fixed value of $Y_1, \dots, Y_s$, the total weight of all such proper tree-structures $\tau$ is (by induction hypothesis) at most $p^Y Q(Y) T_{h-1}(\tau_1) \cdots T_{h-1}(\tau_s) \leq p^Y Q(Y) \mu(Y_1) \cdots \mu(Y_s)$.

By definition, $Y_1, \dots, Y_s$ are distinct and $\{Y_1, \dots, Y_s \} \in \ns(Y)$.  Summing over all such neighbor-sets, we see that 
$$
T_h(Y)  \leq \sum_{ \mathcal T \in \ns(Y) } p^Y Q(Y) \prod_{Y' \in \mathcal T} \mu(Y')
$$

By Theorem~\ref{resample-main-thm}(a), the RHS of this expression is at most $\mu(Y)$.  So we have shown that $T_h(Y) \leq \mu(Y)$ for all integer $h \geq 0$. This shows that $T_{\infty} (Y) \leq \lim_{h \rightarrow \infty} T_h(Y) \leq \lim_{h \rightarrow \infty} \mu(Y) = \mu(Y)$, completing the proof.
\end{proof}

\begin{proof}[Proof of Theorem~\ref{resample-main-thm}]
Let us first show Theorem~\ref{resample-main-thm}(a).  Suppose that the PRA runs for $t$ or more time-steps. By Proposition~\ref{distinct-ts-prop} $\hat \tau^1,  \dots, \hat \tau^{t}$ are distinct, non-empty, appearing tree-structures. By Proposition~\ref{proper-ts} they are all proper. So the number of resamplings is at most the number of appearing proper tree-structures, i.e.
{\allowdisplaybreaks
\begin{align*}
\bE[ \text{\# Resamplings}] &\leq \sum_{\tau \in \Gamma} \Pr( \text{$\tau$ appears}) \leq \sum_{\tau \in \Gamma} w(\tau) \qquad \text{(by Lemma~\ref{couple-lemma})} \\
&= \sum_{Y \in \mathcal A} \sum_{\tau \in \Gamma(Y)} w(\tau) \leq \sum_{Y \in \mathcal A} \mu(Y) \qquad \text{(by Proposition~\ref{wt-bound-prop})}
\end{align*}
}

In particular, since this is finite, the algorithm terminates with probability one.

We derive Theorem~\ref{resample-main-thm}(b) by using the following method to enumerate neighbor-sets $\mathcal T$. First, we put into $\mathcal T$ either one element $Y''$ with $Y'' \bowtie Y$, or no such elements; this contributes a factor $\Bigl( 1 + \sum_{Y'' \bowtie Y} \mu(Y'') \Bigr)$.  Next, we may place any $Y' \sim Y$ into $\mathcal T$. For each $Y'\sim Y$, this contributes the term $1 + \mu(Y')$. As every neighbor-set $\mathcal T$ is generated in this way, we have
$$
\sum_{\mathcal T \in \ns(Y) } \prod_{Y' \in \mathcal T} \mu(Y') \leq \Bigl( \prod_{Y' \sim Y} (1 + \mu(Y')) \Bigr) \Bigl( 1 + \sum_{Y'' \bowtie Y} \mu(Y'') \Bigr)
$$

Finally, we derive Theorem~\ref{resample-main-thm}(c) by setting $\mu(Y) = e p^Y Q(Y)$ for all $Y$. This satisfies Theorem~\ref{resample-main-thm}(b), as any $Y \in \mathcal A$ has
\begin{align*}
  \Bigl( \prod_{Y' \sim Y} (1 + \mu(Y')) \Bigr) \Bigl( 1 + \sum_{Y'' \bowtie Y} \mu(Y'') \Bigr) &\leq e^{\sum_ {Y' \sim Y} \mu(Y')} e^{ \sum_{Y'' \bowtie Y} \mu(Y'')} = e^{\sum_ {Z \approx Y} e p^{Z} Q(Z)} \leq e^{e D P} \leq e
\end{align*}

Therefore, we have
$$
p^Y Q(Y) \Bigl( \prod_{Y' \sim Y} (1 + \mu(Y')) \Bigr) \Bigl( 1 + \sum_{Y'' \bowtie Y} \mu(Y'') \Bigr) \leq e p^Y Q(Y) = \mu(Y)
$$
satisfying Theorem~\ref{resample-main-thm}(b).
\end{proof}

\section{Extension to complex bad-events}
\label{complex-bad-events}
Many applications of the LLL involve multiple bad events that may be more complex than pure atomic events; for example, they may be defined by linear threshold functions. We can always write a complex bad event as a union of a (possibly large) number of atomic bad-events. For example, a bad-event of the form $[X_{i_1} = j_1] + \dots + [X_{i_r} = j_r] \geq  t$ can be represented as $\binom{r}{t}$ separate atomic events.  Thus, we suppose that there are multiple bad-events $\mathcal B_1, \dots, \mathcal B_K$, where each $\mathcal B_k$ is a subset of $\mathcal A$; the sets $\mathcal B_1, \dots, \mathcal B_K$ are not necessarily disjoint. 

It would be natural to apply Theorem~\ref{resample-main-thm} directly on the bad-event set $\mathcal B = \mathcal B_1 \cup \dots \cup \mathcal B_K$. However, there are two technical obstacles to this. The first difficulty is that Theorem~\ref{resample-main-thm} requires us to bound the connection (in terms of the relation $\approx$) between subsets of events of $\mathcal B_k$ and $\mathcal B_\ell$ for $k \neq \ell$. The linkages due to $\sim$ are relatively easy to handle in this way, and are similar to the usual asymmetric LLL. But the linkages due to $\bowtie$ are much more difficult to enumerate and analyze.

In order to avoid this first problem, we will ``decouple'' $\mathcal B_1, \dots, \mathcal B_K$; we analyze each $\mathcal B_k$ separately, deriving an appropriate fractional hitting-set $Q_k$ and computing an appropriate potential function for $\mathcal B_k$. We then sum this potential function over $k = 1, \dots, K$.  In particular, we never need to analyze $\bowtie$-interactions between $\mathcal B_k, \mathcal B_{\ell}$. 

The second technical difficulty is that Theorem~\ref{resample-main-thm} requires defining an auxiliary function $\mu$, and checking a condition on it, for every $Y \in \mathcal A$. Because of our method for decoupling the bad-events, we would actually need to check a separate condition for every $Y \in \mathcal A$  as well as each $k = 1, \dots, K$. This results in a huge inflation in the number of parameters. In order to avoid this, we reparametrize in terms of a more compact auxiliary function, consisting of just variables $\lambda_{i,j}$ for each \emph{element} $(i,j) \in \mathcal X$. We describe how to encode $\mu$ as a function  of $\lambda$, and how to check a much more manageable set of conditions on it.

\subsection{Decoupling bad-events}
Given fractional hitting-sets $Q_1, \dots, Q_K$ for $\mathcal B_1, \dots, \mathcal B_K$ respectively, we will use the following Algorithm~\ref{pra-multi-alg}. This is a slight modification of the PRA given in Section~\ref{pra-first}.

\begin{algorithm}[H]
\caption{$\text{PRA-multi}(p, \mathcal B_1, \dots, \mathcal B_K, Q_1, \dots, Q_K, \mathcal X) $}
\label{pra-multi-alg}
\begin{algorithmic}[1]
\STATE Draw the values $X_1, \dots, X_n$ independently according to the probability distribution $p$.
\WHILE{there is some true bad-event $B \in \mathcal B_1 \cup \dots \cup \mathcal B_K$}
\STATE Arbitrarily select some index $k \in [K]$ and some $B \in \mathcal B_k$ such that $B$ is true.
\STATE Select exactly one subset $Y \subseteq B$. The probability of selecting a given $Y$ is given by
$$
\Pr(\text{select $Y$}) = \frac{ Q_k(Y) }{\sum_{Z \subseteq B} Q_k(Z)}
$$
\STATE Resample every variable involved in $Y$ independently according to probability distribution $p$.
\ENDWHILE
\STATE Return $X$
\end{algorithmic}
\end{algorithm}

We analyze PRA-multi by coupling it to the original version of the PRA when run on a larger set of variables and an appropriate set of bad-events $\tilde {\mathcal B}$. In effect, we encode $\mathcal B_1, \dots, \mathcal B_K$ so that $\tilde {\mathcal B}$ represents the \emph{disjoint} union of $\mathcal B_1, \dots, \mathcal B_K$. 

This new problem instance has the same $n$ variables as the original instance, plus $K |\mathcal A|$ new indicator variables, which we will index as $\langle Y, k \rangle$ for $k \in [K]$ and $Y \in \mathcal A$.  We set $F_{\langle Y, k \rangle} = \{ 0 \}$, i.e.,  each new variable can only take on a single value. Thus, the new set of elements  $\tilde {\mathcal X}$   is given by
$$
\tilde {\mathcal X} = \mathcal X  \cup \Big \{ ( \langle Y, k \rangle, 0) \mid k \in [K], Y \in \mathcal A \Big \} 
$$

To simplify notation in the construction, let us define $[Y,k]$ to denote the singleton set  $[Y,k] =  \{ (\langle Y, k \rangle,0) \}$. We construct the new set of bad-events $\tilde{ \mathcal B}$ as:
$$
\tilde {\mathcal B} = \Bigl \{ B \cup \bigcup_{Y \in \mathcal A} [Y,k] \mid B \in \mathcal B_k \Bigr \}
$$

The following definitions will be used throughout our construction.  For any $Y \in \mathcal A$ and $k \in [K]$, let us define
\begin{equation}
\label{crosseqn}
(Y,k) = Y \cup [Y,k] = Y \cup \{ (\langle Y, k \rangle, 0) \}
\end{equation}

\begin{definition}[\textbf{Good set}] 
We say that a set $Z \subseteq \tilde {\mathcal X}$ is \emph{good} if it has the form $Z =  (Y,k)$ where $Y \in \mathcal A$ and $k \in [K]$.  We say that a set $\mathcal Z \subseteq 2^{\tilde {\mathcal X}}$ is good if every member $Z \in \mathcal Z$ is good.
\end{definition}

Using the notation (\ref{crosseqn}), we often identify the collection of good sets with the space $\mathcal A \times [K]$. (Note that every good set has a \emph{unique} representation as $(Y,k)$).

With these definitions, we construct the corresponding fractional hitting-set $\tilde Q$ for $\tilde {\mathcal B}$ by
$$
\tilde Q(Z) = \begin{cases}
Q_k(Y) & \text{if $Z$ is a good set of the form $Z = (Y,k)$} \\
0 & \text{otherwise}
\end{cases}
$$

We next show that the original PRA on $\tilde{\mathcal B}$ is equivalent to the variant PRA on $\mathcal B$.
\begin{proposition}
\label{decoupling-tog-prop}
Consider running PRA-multi on input $p, \mathcal B_1, \dots, \mathcal B_K, Q_1, \dots, Q_K, \mathcal X$, as well as running the original PRA on input $\tilde p, \tilde{\mathcal B}, \tilde Q, \tilde {\mathcal X}$, where we define $\tilde p$ by
$$
\tilde p_{ij} =
\begin{cases}
p_{ij} & \text{if $i \in [n]$} \\
1 & \text{if $i = \langle Y, k \rangle$}
\end{cases}
$$

With appropriate choices for the resampling rule in the two algorithms,  the probability distribution on the values of $X_1, \dots, X_n$ after $t$ time steps is the same for the two algorithms.
\end{proposition}
\begin{proof}
We use a coupling construction where we run both the original and variant PRA in parallel, taking a common source of randomness for the two algorithms, such that the variables $X_1, \dots, X_n$ have the same value at each time $t$. For PRA-multi, there are additional variables $X_{\langle Y, k \rangle}$; but these always have the value $0$. At time $t = 0$, the two algorithms obviously agree on $X$, since the first step of each is to sample $X_1, \dots, X_n$ from the probability distribution $p$. (PRA-multi also samples, trivially, the variables $X_{\langle Y, k \rangle}$.)

Consider now some time $t > 0$. Let us first suppose that in the original PRA, all the bad-events $\tilde B$ are false. Since  $X_{\langle Y, k \rangle} = 0$ for every $Y \in \mathcal A$, it must be that every bad-event $B \in \mathcal B_k$ is false as well. So, in this case, the two algorithms both terminate and return the common vector $X$.

Now suppose that in the original PRA, the violated set at time $t$ is $\tilde B$,  where $\tilde B =  B \cup \bigcup_{Y \in \mathcal A} [Y,k]$ and $B \in \mathcal B_k$. This implies that $B$ holds on $X_1, \dots, X_n$. PRA-multi may correspondingly use the resampling rule of selecting $k, B \in \mathcal B_k$ in step (3).

In step (4) of the original PRA, we select some $\tilde Y \subseteq \tilde B$ with probability proportional to $\tilde Q( \tilde Y)$. Since $\tilde Q$ is only supported on good sets, we must have $\tilde Y = (Y, k')$ for some $k' \in [K], Y \in \mathcal A$. Since $Y, B \in \mathcal A$ we must have $Y \subseteq B$. Furthermore, since the only entries of $\tilde B$ are $B$ and some sets $[Y',k]$, we must have $k = k'$. Thus, $\tilde Y = (Y, k)$ for some $Y \subseteq B$. Then $\tilde Q( \tilde Y) = Q_k(Y)$.  So step (4) of the original PRA can be viewed as selecting a set $(Y,k)$ where $Y \subseteq B$ in which we a given $Y$ is chosen with probability proportional to $Q_k(Y)$.

Finally, in step (5) of the original PRA, we resample $X_i$ for every $i \sim Y$ (as well as the variable $X_{\langle Y, k \rangle}$); this is precisely what is done in step (5) of PRA-multi. So, if we use the same random bits for the two resamplings, then the values of $X$ agree at time $t+1$ as well.
\end{proof}

In light of Proposition~\ref{decoupling-tog-prop}, we need to satisfy Theorem~\ref{resample-main-thm} for the new problem instance. The following results translate the notations between the sets $\mathcal B_1, \dots, \mathcal B_K$ and their encoding into $\tilde {\mathcal B}$. These proofs are omitted.

\begin{definition}[\textbf{Symmetric relation $\bowtie_k$}]
Given $Y, Y' \in \mathcal A$, we say that $Y \bowtie_k Y'$ if $Y \not\sim Y'$ and there is some $B \in \mathcal B_k$ with $Y, Y' \subseteq B$
\end{definition}

\begin{proposition}
For pairs $(Y, k), (Y', k')$ where $Y, Y' \in \mathcal A$ and $k, k' \in [K]$ the following hold:
\begin{enumerate} 
\item $(Y,k) \sim (Y', k')$ iff $Y \sim Y'$.
\item $(Y,k) \bowtie (Y', k')$ iff $k = k'$ and $Y \bowtie_k Y'$.
\end{enumerate}
\end{proposition}

We say that $(Y,k)$ is \emph{supported} if $Q_k(Y) > 0$. We say that $\bowtie$ is \emph{null} if  $Y \not \bowtie_k Y'$ for all $k \in [K]$, and all sets $Y, Y'$ such that $Q_k(Y) > 0, Q_k(Y') > 0$.

\begin{proposition}
For any $Y \in \mathcal A$ and $k \in [K]$, the set $\mathcal T \subseteq \mathcal A \times [K]$ is a good neighbor-set of $(Y,k)$ (and we write $\mathcal T \in \gns(Y,k)$) if the following conditions hold:
\begin{enumerate}
\item Every $(Z,\ell) \in \mathcal T$ has either (i) $Z \sim Y$ or (ii) $\ell = k$ and $Z \bowtie_k Y$.
\item There do not exist distinct $(Z, k), (Z', k') \in \mathcal T$ with $Z \sim Z'$.
\item There do not exist distinct $(Z,k), (Z', k) \in \mathcal T$ with $Z \bowtie_k Y, Z' \bowtie_{k} Y$.
\end{enumerate}
\end{proposition}

We can now state our main theorem, translated into the new encoding:
\begin{theorem}[\textbf{Main Theorem for multiple events}]
\label{resample-main-thm-2}
Let $Q_1, \dots, Q_K$ be fractional hitting-sets for $\mathcal B_1, \dots, \mathcal B_K$ respectively. In each of the following three cases, PRA-multi terminates in a configuration avoiding all bad events with probability one. 

\smallskip \noindent 
(a) Suppose that $\mu: \mathcal A \times [K] \rightarrow [0, \infty)$ satisfies, for all $Y \in \mathcal A$ and all $k \in [K]$, 
$$
\mu(Y, k) \geq p^Y Q_k(Y) \sum_{\mathcal T \in \gns(Y,k)} \prod_{(Y',k') \in T} \mu(Y',k')
$$
Then, the expected number of resamplings is at most $\sum_{(Y,k)} \mu(Y,k)$.

\smallskip \noindent 
(b) Suppose that $\mu: \mathcal A \times [K] \rightarrow [0, \infty)$ satisfies, for all $Y \in \mathcal A$ and all $k \in [K]$, 
$$
\mu(Y, k) \geq p^Y Q_k(Y) \Bigl( \prod_{Y' \sim Y} (1 + \sum_{k' \in [K]} \mu(Y', k')) \Bigr) \Bigl( 1 + \sum_{Y'' \bowtie_k Y} \mu(Y'', k)  \Bigr)
$$
Then, the expected number of resamplings is at most $\sum_{(Y,k)} \mu(Y,k)$.

\smallskip \noindent 
(c) Suppose that  $p^Y Q_k(Y) \leq P$ for all $Y \in \mathcal A,  k \in [K]$; and suppose that for all supported pairs $(Y,k)$ there are \emph{at most} $D$ supported pairs $(Y', k') \approx (Y,k)$. And suppose finally that $e P D \leq 1$. Then, the expected number of resamplings is at most $e \sum_{(Y,k)} p^Y Q_k(Y)$.
\end{theorem}
\begin{proof}
As in Theorem~\ref{resample-main-thm}, it suffices to prove case (a).  We first claim that $\tilde Q$ is a fractional hitting-set for $\tilde {\mathcal B}$. For, consider some $\tilde B \in \tilde {\mathcal B}$ of the form $\tilde B = B \cup \bigcup_{Y \in \mathcal A} [Y,k]$ for $B \in \mathcal B_k$. Then
$$
\sum_{Y \subseteq \tilde B} \tilde Q(Y) \geq \sum_{Y \subseteq B} \tilde Q(Y \cup [Y,k]) = \sum_{Y \subseteq B} Q_k(Y) \geq 1
$$

Next, in order to apply Theorem~\ref{resample-main-thm}(a), define the function $\tilde \mu: \tilde{\mathcal A} \rightarrow [0, \infty)$ by
$$
\tilde \mu(Z) = \begin{cases}
\mu(Z) & \text{if $Z$ is good} \\
0 & \text{otherwise}
\end{cases}
$$

We want to show that every $Z \in \tilde {\mathcal A}$ has
\begin{equation}
\label{ty1}
\tilde \mu(Z) \geq \tilde p^Z \tilde Q(Z) \sum_{ \mathcal T \in \ns(Z) } \prod_{Y' \in \mathcal T} \tilde \mu(Y')
\end{equation}

If $\mathcal T \in \ns(Z)$ is not good, then $\prod_{Y' \in \mathcal T} \tilde \mu(Y') = 0$. If $Z$ is not good, then (\ref{ty1}) holds easily, as $\tilde Q(Z) = 0$. So suppose $Z = (Y,k)$, and so (\ref{ty1}) reduces to 
$$
\mu(Y,k) \geq p^Y Q_k(Y) \sum_{ \mathcal T \in \gns(Z) } \prod_{(Y', k') \in \mathcal T} \mu(Y',k')
$$
which holds by hypothesis.

So the expected number of resamplings is $\sum_{Z \in \tilde{\mathcal A}} \tilde \mu(Z) = \sum_{Y \in \mathcal A} \sum_{k \in [K]} \mu(Y,k)$. Since this is finite, the PRA terminates with probability one.
\end{proof}

We note that Theorem~\ref{resample-main-thm-2} is equivalent to Theorem~\ref{resample-main-thm} if $K = 1$. (Each time a set $Y \in \mathcal A$ is referenced, simply replace it with $(Y,1)$.)

\subsection{Parametrizing by $\lambda$} 
We now describe a criterion for PRA or PRA-multi in terms of a vector $\lambda \in [0, \infty)^{\mathcal X}$, instead of the function $\mu: \mathcal A \times [K] \rightarrow \mathbf R$. This is a huge savings in terms of the number of parameters.  In fact, the vector $\lambda$ not only encodes $\mu$, but also the probability vector $p$. The vector $\lambda$ should be thought of as an ``inflated'' version of $p$; roughly speaking, $\lambda_{i,j}$ is the probability that $X_i = j$ at some point during the execution of the PRA.

Given the parameter $\lambda$, we define a few related quantities which will be needed to state our theorem.
\begin{definition}[\textbf{Values $\lambda_i, G_i, S$ that depend on a function $Q$ and a vector $\lambda$}] 
\label{defn:G}
Let $\mathcal B \subseteq \mathcal A$, let $\lambda \in [0, \infty)^{\mathcal X}$ and let $Q$ be a fractional hitting-set for $\mathcal B$.

Recalling the notation (\ref{eqn:power-notation}), define
$$
S(\mathcal B, Q, \lambda) = \max_{Q(Y) > 0} \sum_{Z: Z \bowtie Y} Q(Z) \lambda^{Z}
$$
where the definition of $\bowtie$ is with respect to $\mathcal B, Q$.

Also, for each $i \in [n]$, define 
$$
\lambda_i = \sum_j \lambda_{i,j}, \qquad G_i (Q, \lambda) = \sum_{Y \sim i} Q(Y) \lambda^Y
$$
\end{definition}

In the context of multiple bad-events $\mathcal B_1, \dots, \mathcal B_K$, we often write $S_k = S(\mathcal B_k, Q_k, \lambda)$ for simplicity. 

We next state our main theorem in terms of $\lambda$. The correspondence between the $\lambda$ notation and the parameters $p, \mu$ will be given by the following formulas:
$$
p_{ij} = \frac{\lambda_{ij}}{\lambda_i} \qquad \qquad \mu(Y,k) = \frac{\lambda^Y Q_k(Y)}{1 - S(\mathcal B_k, Q_k, \lambda)}
$$
\begin{theorem}[\textbf{Main Theorem in terms of $\lambda$}]
\label{resample-main-thm2}
Let $Q_1, \dots, Q_K$ be fractional hitting-sets for $\mathcal B_1, \dots, \mathcal B_K$ respectively and let $\lambda \in [0,\infty)^{\mathcal X}$. If $S(\mathcal B_k, Q_k, \lambda) < 1$ for all $k \in [K]$, and for all $i \in [n]$ we have
$$
\lambda_i \geq  1 + \sum_k \frac{G_i(Q_k, \lambda)}{1 - S(\mathcal B_k, Q_k, \lambda)}
$$
then PRA terminates with probability one; the expected number of resamplings is at most $\sum_i (\lambda_i-1)$.
\end{theorem}

Before we prove Theorem~\ref{resample-main-thm2}, we record some preliminary calculations which occur in a number of places:
\begin{proposition}
\label{resample-main-thm3}
Assuming that the conditions of Theorem~\ref{resample-main-thm2} are satisfied:
\begin{enumerate}
\item For any $i \in [n]$, we have $\sum_{Y \sim i} \sum_k \mu(Y,k)  \leq \lambda_i - 1$.
\item For any supported $(Y,k)$, we have $\sum_{Z \bowtie_k Y} \mu(Z, k) \leq \frac{S_k}{1 - S_k}$.
\end{enumerate}
\end{proposition}
\begin{proof}
For the first result:
\begin{align*}
\sum_{Y \sim i} \sum_{k} \mu(Y, k) &= \sum_{Y \sim i, k} \frac{\lambda^{Y} Q_k(Y)}{1 -  S_k} = \sum_{k} \frac{G_i(Q_{k}, \lambda)}{1 - S_k} \leq \lambda_i - 1
\end{align*}

For the second result:
\[
\sum_{Z \bowtie_k Y} \mu(Y, k) = \sum_{Z \bowtie_k Y} \frac{\lambda^{Z} Q_k(Z)}{1 -  S_{k}} \leq \frac{S_k }{1 -  S_k} \qedhere
\]
\end{proof}

\begin{proof}[Proof of Theorem~\ref{resample-main-thm2}]
We want to show that $\mu(Y,k)$ satisfies Theorem~\ref{resample-main-thm-2}(a) foir a given $Y \in \mathcal, k \in [K]$, that is,
\begin{equation}
\label{x0}
\mu(Y,k) \geq p^Y Q_k(Y) \sum_{\mathcal T \in \gns(Y,k)} \prod_{(Y',k') \in \mathcal T} \mu(Y',k')
\end{equation}

We now bound the RHS of (\ref{x0}). First, $\mathcal T$ may contain at most one pair $(Z,k)$ with $Z \bowtie_k Y$. By Proposition~\ref{resample-main-thm3}, the total contribution of such terms is at most $1 + \frac{S_k}{1 - S_k} = \frac{1}{1 - S_k}$. Next, for each $(i,j) \in Y$, the set $\mathcal T$ may contain at most one pair $(Y', k')$ with $i \sim Y'$. (If there two such pairs $(Y', k'), (Y'', k'')$, then $(Y',k) \sim (Y'', k'')$.) For any fixed $(i,j)$ in $Y'$, this contributes at most $1 + \sum_{Y' \sim i} \sum_{k'} \mu(Y', k')$; by Proposition~\ref{resample-main-thm3} this is at most $1 + (\lambda_i - 1) = \lambda_i$.

Putting these two estimates together,  we estimate the RHS of (\ref{x0}):
{\allowdisplaybreaks
\begin{align*}
&p^Y Q_k(Y) \sum_{\mathcal T \in \gns(Y,k)} \prod_{(Y',k') \in T} \mu(Y',k') \leq p^Y Q_k(Y) \times \frac{1}{1-S_k} \times  \prod_{(i,j) \in Y} \lambda_i  \\
&\qquad= \frac{Q_k(Y)  \prod_{(i,j) \in Y} p_{ij} \lambda_i}{1 - S_k} = \frac{Q_k(Y) \prod_{(i,j) \in Y} \frac{\lambda_{ij}}{\lambda_i} \lambda_i}{1 - S_k}  = \frac{Q_k(Y) \lambda^Y}{1 - S_k} = \mu(Y,k)
\end{align*}
}

Thus, Theorem~\ref{resample-main-thm-2}(a) holds. Note that since $\mathcal A$ contains only non-empty sets, every $Y \in \mathcal A$ has $Y \sim i$ for at least one value $i \in [n]$. Therefore, using Proposition~\ref{resample-main-thm3}, we compute
\[
\bE[\text{\# resamplings}] \leq \sum_{Y,k} \mu(Y, k) \leq \sum_{i \in [n]} \sum_{Y \sim i} \sum_k \mu(Y,k) \leq \sum_{i \in [n]} (\lambda_i-1)  \qedhere
\]
\end{proof}

When applying Theorem~\ref{resample-main-thm2}, we note that if $\bowtie_k$ is null, then $S_k = 0$. Also, we frequently use the crude bound:
\begin{equation}
\label{crude-sk-bound}
S(\mathcal B, Q, \lambda) \leq \sum_{Y \in \mathcal A} Q(Y) \lambda^Y
\end{equation}

The parametrization by $\lambda$ can be useful for the standard MT algorithm (as the MT algorithm is a special case of the PRA). In particular, it gives clean formulas for the ``MT distribution,'' which is the distribution on the variables at the termination of the MT algorithm. See Appendix~\ref{mt-dist-appendix} for further discussion and some examples.

\section{Bad-events defined by sums of random variables}
\label{sec:random-sum}
In this section, we explore a connection between symmetric polynomials and Chernoff tail bounds for sums of indicator variables of elements (i.e. terms of the form $[X_i = j]$). These will be central to assignment-packing problems. When the bad-events are defined in terms of such sums, then these symmetric polynomials correspond in a natural way with fractional hitting-sets.  

To state these results in their clearest form, we define the Chernoff upper-tail separation function:
\begin{definition}[\textbf{Chernoff upper-tail separation function $\text{Ch}(\mu,t)$}]
For real numbers $\mu, t > 0$, we define
$$
\text{Ch}(\mu, t) = e^{t - \mu} (\mu/t)^t
$$
\end{definition}

We recall three useful results of \cite{sss} on multivariate symmetric polynomials and Chernoff bounds.
\begin{proposition}
\label{sym-pol-bound1}
For any real numbers $a_1, \dots, a_{\ell} \in [0,1]$ and integer $k \leq a_1 + \dots + a_{\ell}$, we have
$$
\sum_{\substack{X \subseteq [\ell] \\  |X| = k}} \prod_{x \in X} a_x \geq \binom{ a_1 + \dots + a_{\ell}}{k}
$$
\end{proposition}

\begin{proposition}
\label{sym-pol-bound2}
For any real numbers $a_1, \dots, a_{\ell} \in [0,\infty)$ and integer $k \geq 0$, we have
$$
\sum_{\substack{X \subseteq [\ell] \\  |X| = k}} \prod_{x \in X} a_x \leq \frac{(a_1 + \dots + a_{\ell})^k}{k!}
$$
\end{proposition}

\begin{proposition}
\label{ch-d-thm}
For $0 \leq \mu \leq t$ and $d = \lceil t - \mu \rceil$, we have $\frac{\mu^d} {d! \binom{t}{d}}  \leq \text{Ch}(\mu, t).$
\end{proposition}

We can now state our main result which transforms a bad-event defined by sums of random variables into a fractional hitting set.
\begin{theorem}
\label{devthm}
Let $\lambda \in [0, \infty)^{\mathcal X}$ and $a \in [0,1]^{\mathcal X}$ be two vectors of real numbers,  indexed by the elements $x = (i,j)$ of $\mathcal X$. Define $\mu = \sum_x a_x \lambda_x$, and for each $i \in [n]$ define $\mu_i  = \sum_{j} a_{i,j} \lambda_{i,j}$. For a real number $t \geq 0$,  let $\mathcal B$ be the complex bad-event defined by
$$
\mathcal B \equiv \sum_{i,j} a_{i,j} [X_i = j] \geq t
$$

Let $d$ be any integer in the range $1 \leq d \leq t$. Then, recalling Definition~\ref{defn:G}, there is a fractional hitting-set $Q$ with the property 
\[  S(\mathcal B, Q, \lambda) \leq \frac{\mu^d} {d! \binom{t}{d}}; ~~
G_i(Q, \lambda) \leq \frac{d \mu_i}{\mu}  \cdot \frac{\mu^{d}}{d! \binom{t}{d}}. \]

\end{theorem}
\begin{proof}
We define $Q$ by 
$$
Q(Y) = \begin{cases}
\frac{\prod_{x \in Y} a_x}{\binom{t}{d}} & \text{if $|Y| = d$} \\
0 & \text{otherwise}
\end{cases}
$$

To show that this is a valid fractional hitting-set, consider some atomic bad-event $B$, where $\sum_{x \in B} a_x \geq t$. By Proposition~\ref{sym-pol-bound1}, we have
$$
\sum_{Y \subseteq B} Q(Y) = \frac{\sum_{Y \subseteq B, |Y| = d} a^Y}{\binom{t}{d}} \geq \frac{ \binom{\sum_{x \in B} a_x}{d} }{\binom{t}{d}} \geq \frac{ \binom{t}{d}}{\binom{t}{d}} = 1.
$$

We use Proposition~\ref{sym-pol-bound2} and (\ref{crude-sk-bound}) to compute $S(\mathcal B, Q, \lambda)$ as:
$$
S(\mathcal B, Q, \lambda) \leq \sum_{Y \in \mathcal A} \lambda^Y Q(Y) =\frac{\sum_{Y \in \mathcal A, |Y| = d} \lambda^Y a^Y}{\binom{t}{d}} \leq \frac{ (\sum_{x \in Y} a_x \lambda_x)^d }{d! \binom{t}{d}}  =\frac{ \mu^d  }{d! \binom{t}{d}}.
$$

Similarly, we use Proposition~\ref{sym-pol-bound2} to compute $G_i(Q, \lambda)$ as:
\begin{align*}
G_i(Q, \lambda) &= \sum_{Y \sim i} \lambda^Y Q(Y) = \frac{\sum_{j} a_{i,j} \lambda_{i,j}}{\binom{t}{d}} \sum_{\substack{Y \in \mathcal A \\ Y \not \sim i \\ |Y| = d-1 }} \lambda^Y  \leq \frac{\sum_{j} a_{i,j} \lambda_{i,j} \Bigl( \sum_{\substack{(u,j) \in \mathcal X \\ u \neq i}} a_{(u,j)} \lambda_{(u,j)} \Bigr)^{d-1}}{(d-1)! \binom{t}{d}} \\
&= \frac{\mu_i (\mu - \mu_i)^{d-1}}{(d-1)! \binom{t}{d}} \leq \frac{d \mu_i}{\mu} \frac{\mu^d}{d! \binom{t}{d}}. \qedhere
\end{align*}
\end{proof}

We note that in order to use the fractional hitting-set $Q$ of Theorem~\ref{devthm} in the PRA, we must be able to efficiently access $Q$.  We discuss how to do this in Appendix~\ref{access-q-sec}.

\begin{corollary}
\label{devthm2}
Let $\lambda \in [0, \infty)^{\mathcal X}$ and $a \in [0,1]^{\mathcal X}$ be two vectors of real numbers,  indexed by the elements $x = (i,j)$ of $\mathcal X$. Define $\mu = \sum_x a_x \lambda_x$, and for each $i \in [n]$ define $\mu_i  = \sum_{j} a_{i,j} \lambda_{i,j}$. For a real number $t \geq 1$,  let $\mathcal B$ be the complex bad-event defined by
$$
\mathcal B \equiv \sum_{i,j} a_{i,j} [X_i = j] \geq t
$$

Let $r$ be any real number satisfying $1 \leq r < t$ and $r \geq \mu$. Then there is a fractional hitting set $Q$ with the property
\[  S(\mathcal B, Q, \lambda) \leq \text{Ch}(r,t); ~~
G_i(Q, \lambda) \leq \mu_i ( \frac{t+1}{r} - 1 ) \text{Ch}(r,t); ~~ \]
\end{corollary}
\begin{proof}
  Apply Theorem~\ref{devthm} with parameter $t$ and $d = \lceil t - r \rceil$. Note that since $r < t$ we have $d \geq 1$, and since $r \geq 1$ we have $d \leq \lceil t - 1 \rceil \leq t$. With this value of $d$, Proposition~\ref{ch-d-thm} gives:
$$
    S(\mathcal B, Q, \lambda) \leq \frac{\mu^d} {d! \binom{t}{d}} \leq \frac{r^d}{d! \binom{t}{d}} \leq \text{Ch}(r,t)
    $$
    and similarly
    \begin{align*}
      G_i(Q, \lambda) &\leq \frac{d \mu_i}{\mu}  \cdot \frac{\mu^{d}}{d! \binom{t}{d}} \leq \frac{d \mu_i}{r}  \cdot \frac{r^{d}}{d! \binom{t}{d}} \leq \frac{(t-r+1) \mu_i}{r}  \cdot \frac{r^{d}}{d! \binom{t}{d}} \leq \mu( \frac{t+1}{r} - 1  ) \text{Ch}(r,t)      \qedhere
    \end{align*}    
  \end{proof}

\subsection{Column-sparse assignment-packing problems}
We now consider a family of constraint satisfaction problems we refer to as \emph{assignment-packing problems}. Here, we have $m$ linear packing constraints of the form  ``$\sum_{i,j} a_{k,i,j} [X_i = j] \leq b_k$'', where  $a_{k,x} \in [0,1], b_k \geq 0$. We also assume that this CSP is ``column-sparse'', in the sense that for any $x \in \mathcal X$ we have $\sum_k a_{k,x} \leq D$ for some parameter $D$. (As above, $x$ will often refer to some element $(i,j) \in \mathcal X$.) When does such an 
integer linear program have a feasible solution? 

One technique to solve this uses an LP relaxation, where each $(i,j) \in \mathcal X$ has a fractional variable $z_{i,j} \in [0,1]$ to represent  $[X_i = j]$. Each $i \in [n]$ has a linear constraint $ \sum_{j \in F_i} z_{i,j} = 1$. In addition the packing constraints are tightened to
$\sum_x a_{k,x} z_x \leq c_k$ for some $c_k \leq b_k$. We then convert the fractional solution $z$ into an integral solution via some form of randomized rounding. The central problem becomes how close $c_k$ can be to $b_k$, in order to ensure that a feasible integral solution can be converted into a feasible integral solution.

We describe here a randomized rounding algorithm based on a single application of the PRA. This involves a a parameter $\epsilon > 0$, which determines a multiplicative increase in $b_k$ as compared to $c_k$, in addition to a secondary discrepancy term.  Our main result will be the following:
\begin{theorem}
  \label{csp-thm}
  Let us suppose that the linear program
  $$
\sum_{j \in F_i} z_{i,j} = 1, \qquad \sum_x a_{k,x} z_x \leq c_k, \qquad z_x \in [0,1]
$$
is satisfiable, for some vector $c \in [1, \infty)^m$, and that we have a separation oracle for it (that is, given a variable assignment, we can either find a violated linear constraint, or determine that all constraints are satisfied.)
  
Let $D \geq 2$ be some parameter satisfying
$$
D \geq \max_{x} \sum_{k} a_{k,x}
$$
and let $\epsilon$ be a parameter satisfying $0 < \epsilon \leq 1/D$.

Then in expected time $\poly(n)$, we can find a solution to the integer program
$$
X_i \in F_i \qquad \sum_{i,j} a_{k,i,j} [X_i = j] \leq b_k
$$
where we define the vector $b$ by
$$
b_k = \begin{cases}
\frac{100 \ln(1/\epsilon)}{1 + \ln( \frac{\ln(1/\epsilon)}{c_k} )} & \text{for $c_k \leq \ln(1/\epsilon)$} \\
c_k(1+\epsilon) + 10 \sqrt{c_k \ln \bigl( D + \frac{1}{c_k \epsilon^2} \bigr)} & \text{for $c_k > \ln(1/\epsilon)$} 
\end{cases}
$$
\end{theorem}
\begin{proof}
  The full proof requires technical analysis of the Chernoff tail-separation function. In order to separate this routine functional analysis from our analysis of the PRA, we defer some calculations to Appendix~\ref{csp-proof-app}, in which we show that the vector $b$ satisfies the following conditions for all $k \in [m]$:
  \begin{description}
\item[(C1)] $(\frac{b_k+1}{c_k (1 + \epsilon)} - 1) \text{Ch}(c_k (1 + \epsilon), b_k) \leq \frac{\epsilon}{4 D}$
\item[(C2)] $\text{Ch}(c_k (1 + \epsilon), b_k) \leq 1/2$ 
  \end{description}
  
To begin, the separation oracle allows us to solve the LP in polynomial time; let $\hat z$ be the resulting fractional solution. We use the framework of Section~\ref{complex-bad-events}, in which each packing constraint will correspond to a complex bad-event (that is, $K = m$), and where $\lambda = (1 + \epsilon) \hat z$.

Let us analyze a constraint $k$, corresponding to the complex bad-event $\mathcal B_k \equiv \sum_{i,j} a_{k,i,j} [X_i = j] \geq b_k$. The fractional hitting-set $Q_k$ will be the one of Corollary~\ref{devthm2}, with $t = b_k$ and $r = (1+\epsilon) c_k$. Note that $r \geq 1$ since $c_k \geq 1$, and $r < t$ by property (C2). To show $r \geq \mu$ using the terminology of Corollary~\ref{devthm2}, we compute
$$
\mu = \sum_{i,j \in F_i} a_{k,i,j} \lambda_{i,j} = \sum_{i,j \in F_i} a_{k,i,j} (1 + \epsilon) \hat z_{i,j} \leq (1+\epsilon) c_k = r.
$$

We next calculate $G_i$ and $S_k$. By Corollary~\ref{devthm2} and using property (C2), 
$$
  S_k = S(\mathcal B_k, Q_k, \lambda) \leq \text{Ch}(r,t) = \text{Ch}(c_k(1+\epsilon), b_k) \leq 1/2  
$$

Furthermore for each $i \in [n]$, property (C1) gives
{\allowdisplaybreaks
\begin{align*}
G_i(Q_k, \lambda) &\leq  \mu_i (\frac{t +1 }{r} + 1) \text{Ch}(r,t) =\mu_i \Bigl( \frac{b_k+1}{c_k(1+\epsilon) } - 1 \Bigr) \text{Ch}(c_k(1+\epsilon), b_k) \leq \mu_i \frac{\epsilon}{4 D}
\end{align*}
}

We compute $\mu_i$ by:
$$
\mu_i = \sum_{j \in F_i} a_{k,i,j} \lambda_{k,i,j} \leq (1+\epsilon) \sum_j a_{k,i,j} \hat z_{i,j}
$$

To apply Theorem~\ref{resample-main-thm2}, we sum over all $j \in F_i$ obtaining
{\allowdisplaybreaks
\begin{align*}
\sum_{j \in F_i} \lambda_{i,j} - \sum_{k} \frac{G_i(Q_k, \lambda)}{1 - S_k} &\geq (1+\epsilon) - \sum_k \frac{ \epsilon/(4 D) \times (1+\epsilon) \sum_{j} a_{k,i,j} \hat z_{i,j}}{1/2} \\
&= (1+\epsilon) - \tfrac{1}{2}\sum_{j \in F_i} \hat z_{i,j} (1+\epsilon) \times (\epsilon/D) \times \sum_k a_{k,i,j}\\
&\geq (1+\epsilon) - \tfrac{1}{2} \sum_{j \in F_i} \hat z_{i,j} (1+\epsilon) \times (\epsilon/D) \times D  \qquad \text{definition of $D$} \\
&= (1+\epsilon) - \tfrac{1}{2}(1+\epsilon) \epsilon \geq 1 \qquad \text{as $\epsilon \in [0,1]$}
\end{align*}
}

Furthermore, for any $i \in [n]$ we have $\sum_{j \in F_i} \lambda_{i,j} \leq 1+\epsilon \leq 2$, so the expected number of iterations before the PRA terminates is at most $ \sum_{i,j} \lambda_{i,j} \leq O(n)$. Each step of the PRA can be efficiently implemented using the separation oracle, so this gives a polynomial-time algorithm (even though the number of constraints $m$ may be exponential in $n$). 
\end{proof}

We can simplify Theorem~\ref{csp-thm} when we have a uniform bound on the RHS values $c_k$.
\begin{proposition}
  \label{csp-thm3}
  Let us suppose that the linear program
  $$
\sum_{j \in F_i} z_{i,j} = 1, \qquad \sum_x a_{k,x} z_x \leq c_k, \qquad z_x \in [0,1]
$$
is satisfiable, for some vector $c \in [1, R]^m$, and that we have a separation oracle for it.
  
  Let $D \geq 2$ be some parameter satisfying $D \geq \max_{x} \sum_{k,x} a_{k,x}$.

  Then in expected time $\poly(n)$, we can find a solution to the integer program
$$
X_i \in F_i \qquad \sum_{i,j} a_{k,i,j} [X_i = j] \leq b_k
$$
where we define the vector $b$ by
$$
b_k = \begin{cases}
\frac{C \ln D}{1 + \ln( \frac{\ln D}{c_k})} & \text{if $R \leq \ln D$} \\
c_k + C \sqrt{R \ln D} & \text{if $R > \ln D$}
\end{cases}
$$
for some universal constant $C$.
\end{proposition}
\begin{proof}
  Let $d = \ln D$.   If $R \leq d$, then this follows immediately from Theorem~\ref{csp-thm} with $\epsilon = 1/D$. Otherwise, for $R > d$, we apply Theorem~\ref{csp-thm} with $\epsilon = D^{-10} \sqrt{d/R}$; note that $\epsilon \leq 1/D$ in this case.  Let $\delta = \ln(1/\epsilon)= 10 d + \tfrac{1}{2} \ln(R/d)$. If $c_k \leq \delta$, then 
$$
b_k =  \frac{100 \delta}{1 + \ln(\delta/R)} \leq \frac{O(d + \ln(R/d))}{1 + \ln(\delta/R)} \leq O(d + \ln(R/d)) \leq O(d + \log R) \leq O(\sqrt{R \log D})
$$

If $c_k \geq \delta$, we have
\begin{align*}
  b_k &= c_k (1 + \epsilon) + 10 \sqrt{R \ln(D + \frac{1}{c_k \epsilon^2})} = c_k + ( c_k D^{-10} \sqrt{d/R}) + 10 \sqrt{c_k \ln\bigl( D + \frac{D^{20} R}{c_k d} \bigr) }
\end{align*}

As $c_k \leq R$, the term $c_k D^{-10} \sqrt{d/R}$ is at most $R D^{-10} \sqrt{ \frac{\log  D}{R}} \leq O( \sqrt{R})$. Likewise, simple analysis shows that $c_k \ln \Bigl( D + \frac{D^{20} R}{c_k d} \Bigr) \leq O(R \log  D)$, 
giving the claimed result.
  \end{proof}

We note that there is a fundamental problem in how the standard LLL deals with fractional entries in the constraint matrix $A$. The issue is that variable $X_i$ affects constraint $k$ if $a_{k,i,j} > 0$, and it is possible that every variable affects every constraint. In the LLL setting, one cannot quantify ``how much'' $X_i$ affects a constraint. As shown in \cite{harvey}, it is possible to sidestep this issue by quantization of $A$: if $a_{k,i,j}$ is close to zero, then it gets quantized to zero and thus $X_i$ does not affect the given constraint at all. However, this is cumbersome and unnatural.

The LLL has other problems in dealing with CSP's, even when the matrix etntries are in the range $\{0, 1 \}$. In Appendix~\ref{compare-mt-sec}, we compare the PRA to the MT algorithm, with  a rather extreme demonstration of how the PRA is able to make steady progress to the solution, whereas the MT algorithm on its own is completely unable to do so. Specifically, we construct a family of linear threshold constraint-satisfaction problems, in which every entry of the matrix $A$ is in the range $\{0, 1 \}$, but where still \emph{every} variable effects \emph{every} constraint. In this case, there is no ``locality'' in the sense of the LLL. The MT algorithm completely throws away its partial solution at every iteration, and starts from scratch. Not surprisingly, the MT algorithm cannot guarantee \emph{any} scale-free approximation factors (independent of the number of constraints or variables) for this type of problem.  By contrast, the PRA yields a very good approximation in expected polynomial time. 

We also note that that there is a related class of integer programming problems referred as \emph{column-sparse covering integer problems.} A variant of the PRA can be used to solve these systems; see \cite{chen2016partial} for more details.

\section{Applications of column-sparse assignment problems}
\label{sec:column-sparse}
In this section, we discuss two straightforward applications of Theorem~\ref{csp-thm}, to a multi-dimensional scheduling approximation algorithm and to a problem in discrepancy. The results we obtain have not been stated before explicitly; however, it is possible to derive them using a combination of previous rounding algorithms including the LLL, polyhedral rounding such as \cite{klrtvv}, and the algorithm of \cite{llrs}. We include them here to better explain the PRA, and to demonstrate how it gives a simpler and more unified approach to such discrepancy bounds.

\subsection{Multi-dimensional scheduling}
Scheduling on unrelated parallel machines is a classic problem in operations research \cite{LST}. In this setting, we have $n$ jobs and $m$ machines, and each job $i$ needs to be scheduled on some machine.\footnote{We interchange the standard use of the indices $i$ and $j$ here in order to conform to the rest of our notation.} If job $i$ is scheduled on machine $j$, then $j$ incurs a load of $p_{i,j}$. The goal is to minimize the \emph{makespan}, the maximum total load on any machine. The standard way to approach this is to introduce an auxiliary parameter $T$, and ask if we can schedule with makespan $T$ \cite{LST,singh:thesis}. This leads to the integer program formulation:
\begin{eqnarray*}
\forall i \in [n],  ~\sum_{j \in F_i} z_{i,j} & = & 1; \\
\forall j \in [m], ~\sum_i p_{i,j} z_{i,j} & \leq & T; \\
\forall (i,j), ~p_{i,j} > T & \implies & z_{i,j} = 0;
\end{eqnarray*}

Azar \& Epstein \cite{azar-epstein:lp} considered a generalization where  there are $d$ dimensions to the load (say runtime, energy, heat consumption, etc.) when job $i$ gets assigned to machine $j$. In dimension $\ell$, this assignment leads to a load of $p_{i,j,\ell}$ on $j$. We ask here: \emph{given a vector $(T_1, T_2, \ldots, T_d)$, is there an assignment that has a makespan of at most $T_{\ell}$ in each dimension $\ell$?} Azar \& Epstein described a $(d+1)$-approximation algorithm for this. We show how the PRA can improve this to $O(\frac{\log d}{\log \log d})$:
\begin{theorem}
Given a feasible makespan vector $(T_1, \dots, T_d)$, there is a randomized polynomial-time algorithm to find a schedule with makespan vector $(T'_1, \dots, T'_d)$ where
$$
T'_i \leq O \Bigl( T_i \times \frac{\log d}{\log \log d} \Bigr)
$$
\end{theorem}
\begin{proof}
This is an easy application of Proposition~\ref{csp-thm3}. First, we set $z_{i,j} := 0$ if there exists some $\ell$ with $p_{i,j,\ell} > T_{\ell}$. After this filtering, we solve the LP relaxation with the  constraints:
$$
\forall j \in [m], \forall \ell \in [d], ~\sum_i \frac{p_{i,j, \ell}}{T_\ell} z_{i,j} \leq 1;
$$

This LP has its RHS values (i.e. entries of $c_k$) equal to one, and our filtering ensures that the coefficient matrix has entries $a_{i,(j,\ell)} = \frac{p_{i,j,\ell}}{T_\ell} \in [0,1]$. Each element $(i,j) \in \mathcal X$ has $d$ constraints, so the maximum column sum is $D = d$.  By Proposition~\ref{csp-thm3}, we get $b_k = O \bigl( \frac{\log D}{\log \log D} \bigr)$.
\end{proof}

\subsection{Discrepancy}
\label{discrep-app}
As another application, we consider a discrepancy problem introduced in \cite{harvey}, and analyzed there via the LLL. 
\begin{theorem}
Let $Y$ be an $m \times n$ matrix whose entries are real numbers in the range $[-1,1]$, which satisfies the following bounds on the $\ell_1$ norms of the rows and columns:
$$
\forall i \sum_k |Y_{k,i}| \leq D \qquad \forall k \sum_i |Y_{k,i}| \leq R,
$$
for parameters $R \geq 1, D \geq 2$. Then there is a randomized polynomial-time algorithm to find a vector $v \in \{-1, +1 \}^n$  such that, for all $k \in [m]$
$$
| Y_k \cdot v | \leq O( \sqrt{R \log D} )
$$
\end{theorem}
\begin{proof}
Let $d = \ln D$ and for each $k \in [m]$ let $y_k = \sum_i |Y_{k,i}| \leq R$.  If $R \leq d$, then the stated result holds trivially, as $|Y_k \cdot v | \leq y_k \leq R \leq \sqrt{R d}$ for any such vector $v$. So, let us assume $R > d$.

For each $i \in [n]$ we introduce a variable $X_i$ which takes two possible values which we name $+1, -1$. For each $k \in [m]$ we introduce a packing constraint
$$
\sum_{i: Y_{k,i} > 0} Y_{k,i} z_{i,+1} + \sum_{i: Y_{k,i} < 0} (-Y_{k,i}) z_{i,-1} \leq c_k
$$
where $c_k = y_k/2$. This LP has a fractional solution defined by $z_{i,+1} = z_{i,-1} = 1/2$ for all $i$, and has maximum $\ell_1$-column-norm of $D$. We now apply Theorem~\ref{csp-thm3} with parameter $R$ to obtain an integral solution $v_1, \dots, v_n \in \{-1, 1 \}^n$, such that 
$$
\sum_{i: Y_{k,i} > 0} Y_{k,i} [v_i = +1] + \sum_{i: Y_{k,i} < 0} (-Y_{k,i}) [v_i = -1] \leq b_k;
$$
here, we observe that our assumption $R > d$ ensures that $b_k \leq c_k + O(\sqrt{R d}) = y_k/2 + O(\sqrt{R d})$.

This vector $v$ achieves the desired result, as for each constraint $k \in [m]$, we have 
{\allowdisplaybreaks
\begin{align*}
Y_k \cdot v &= \sum_{i: Y_{k,i}>0} (-Y_{k,i} + 2 Y_{k,i} [v_i = +1])  + \sum_{i: Y_{k,i}<0} (Y_{k,i} - 2 Y_{k,i} [v_i = -1]) \\
&= -y_k + 2 \bigl( \sum_{i: Y_{k,i}>0} Y_{k,i} [v_i = +1]  + \sum_{i: Y_{k,i}<0} - Y_{k,i} [v_i = -1] \bigr) \\
&\leq -y_k + 2  \bigl ( y_k/2 + O(\sqrt{R d}) \bigr) \leq O(\sqrt{R d}) \qedhere
\end{align*}
}
\end{proof}

By contrast, \cite{harvey} shows the weaker bound $|Y_k \cdot v| \leq O( \sqrt{R \log(RD)})$.

\section{Transversals with omitted subgraphs}
\label{sec:transversals}
Consider a graph $G = (V,E)$ with a partition of its vertices into sets $V = V_1 \sqcup V_2 \sqcup \dots \sqcup V_\ell$, each of size $b$. We refer to these sets as \emph{blocks} or \emph{classes}. We wish to select exactly one vertex from each block; such a set of vertices $A \subseteq V$ is known as a \emph{transversal}. There is a large literature on selecting transversals such that the graph induced on $A$ omits certain subgraphs. (This problem was introduced in a slightly different form by \cite{bollobas-erdos-szemeredi:isr}; more recently it has been analyzed in \cite{szabo-tardos:extremal,haxell-szabo-tardos:partitioning,yuster:transversals,jin:transversals,haxell-szabo:odd2006}.) For example, when $A$ is an independent set of $G$ (omits the 2-clique $K_2$), this is referred to as an \emph{independent transversal}. 

Alon gives a short LLL-based proof that a sufficient condition for such an independent transversal to exist is $b \geq 2e \Delta$ \cite{alon:lin-arb}, where $\Delta$ is the maximum degree of $G$. The cluster-expansion version of the LLL \cite{bissacot} easily shows that $b \geq 4 \Delta$ suffices. Haxell shows non-constructively  that a sufficient condition is $b \geq 2 \Delta$ \cite{haxell:struct-indep-set}; this condition is existentially optimal, in the sense that $b \geq 2 \Delta - 1$ is not always admissible \cite{jin:transversals,yuster:transversals,szabo-tardos:extremal}. A similar criterion $b \geq \Delta + \lfloor \Delta/r \rfloor$ is given in  \cite{haxell-szabo-tardos:partitioning} for the existence of a transversal which induces no connected component of size greater than $r$. Finally, \cite{locally-sparse} gives a criterion of $b \geq \Delta$ for the existence of a transversal omitting $K_3$; this is the optimal constant but the result is non-constructive.

\subsection{Avoiding large cliques}
Let us consider the problem of finding an independent transversal omitting $K_s$. We will be interested in the case when both $s$ and $\Delta$ are large. More specifically, for any value of $s$ we seek a bound of the form $b \geq \gamma_s \Delta$ (where $b, \Delta$ may go to infinity). We then seek to understand the behavior of the value of $\gamma_s$ as $s \rightarrow \infty$. 

We must have $\gamma_s \geq 1/(s-1)$. To see this, note that for any value of $b \geq 1$, we may take a graph consisting of $s$ blocks each containing $b$ vertices, where every vertex is connected to all vertices outside its block. This graph has $\Delta = b (s-1)$, and clearly any transversal contains a copy of $K_s$. An argument of \cite{szabo-tardos:extremal} shows the slightly stronger lower bound $\gamma_s \geq \frac{s}{(s-1)^2}$; intriguingly, \cite{szabo-tardos:extremal} also conjecture this to be exactly tight. On the other hand, a non-constructive argument of \cite{locally-sparse} shows that $\gamma_s \leq 2/(s-1)$; this is the best general upper-bound on $\gamma_s$ previously known. 

The following result shows that the lower-bound of \cite{szabo-tardos:extremal} gives the correct \emph{asymptotic} rate of growth, up to lower-order terms. 
\begin{theorem}
\label{thm:avoidclique0}
For $s \geq 1$,  we have $\gamma_s \leq \frac{1}{s} \Bigl( 1 + O( \frac{1}{\sqrt{s}} ) \Bigr)$.
\end{theorem}
\begin{proof}
We define a variable $X_i$ for each block $i$, wherein $X_i$ is the choice of which vertex in block $i$ goes into the transversal, and we use the probability distribution setting $p_{iv} = 1/b$ for each block $i$ and vertex $v$ in that block. 

We have a separate bad-event for each $s$-clique. We define $Q$ by setting $Q(Y) = 1/\binom{s}{r}$ whenever $Y$ corresponds to an $r$-clique in the graph, where $r < s$ is some parameter to be chosen. This satisfies the definition of a fractional hitting set, since an $s$-clique contains exactly $\binom{s}{r}$ $r$-cliques.

We will apply Theorem~\ref{resample-main-thm}(c) to show that the PRA terminates with a good configuration as long as $b \geq (\Delta/s) (1 + \frac{c}{\sqrt{s}})$ and $c$ is some sufficiently large universal constant. 

For any $r$-clique $Y$, we have $p^Y Q(Y) = \frac{(1/b)^r}{\binom{s}{r}}$. We need to count how many other $r$-cliques $Z$ have $Z \approx Y$. First, to enumerate all $Z \sim Y$, we may select any vertex $v \in Y$, select another vertex $u$ from the block of $v$, and any choices of $r-1$ neighbors of $u$. Thus, there are at most $b r \binom{\Delta}{r-1}$ choices of $Z$ with $Z \sim Y$. To count the number of $r$-cliques $Z$ with $Z \bowtie Y$, note that if we fix some vertex $v \in Y$, then every vertex in $Z$ is a neighbor of $v$ (since $Y, Z$ are subsets of a common $s$-clique). So there are are most $\binom{\Delta}{r}$ choices for $Z$.

So, we apply Theorem~\ref{resample-main-thm}(c) with $P = \frac{1}{b^r \binom{s}{r}}$ and $D = b r \binom{\Delta}{r-1} + \binom{\Delta}{r}$. We calculate
$$
e P D = \frac{e (b r \binom{\Delta}{r-1} + \binom{\Delta}{r})}{b^r \binom{s}{r}} \leq \frac{e ( r^2 (b/\Delta) + 1) }{(b/\Delta)^r \binom{s}{r} r!}
$$

Simple calculus shows that $e P D \leq 1$ is satisfied for $s$ sufficiently large with $r = \lceil \sqrt{s} \rceil$ and $b/\Delta = s^{-1} + 3 s^{-3/2}$. This implies that for $s$ sufficiently large and $b \geq \Delta (s^{-1} + s^{-3/2})$, the PRA will find a transversal omitting $K_s$. In particular, $\gamma_s \leq s^{-1} + 3 s^{-3/2}$ for $s$ sufficiently large.
\end{proof}

By way of comparison, let us consider bounding $\gamma_s$ via the standard LLL. Each $s$-clique $H$ of the graph corresponds to a bad-event. Each such bad event has probability $(1/b)^s$, and affects at most $s b \Delta^{s-1}/(s-1)!$ other $s$-cliques. The symmetric LLL criterion is therefore satisfied when
$$
e \times (1/b)^s \times s b \Delta^{s-1} / (s-1)! \leq 1
$$
which leads to the condition $b/\Delta \geq \Bigl( \frac{e s}{(s-1)!} \Bigr)^{\frac{1}{s-1}} = e/s + o(1/s)$, which is worse by a constant factor.

Although Theorem~\ref{thm:avoidclique0} shows that the PRA terminates with a configuration avoiding $s$-cliques, this does not lead to an efficient algorithm. The reason is that in order to implement the PRA, we must detect if there is some true bad-event; this would require finding a clique in the graph, which is intractable.  In order to obtain a fully constructive algorithm, we must enforce a \emph{stronger} (but easy-to-check) condition on the graph: not only do we avoid copies of $K_s$, but we in fact avoid $(s-1)$-stars. This leads to a bound on $b$ with slightly weaker second-order terms.

\begin{theorem}
\label{thm:avoidclique}
Let $G$ be a graph of maximum degree $\Delta$ whose vertex set is partitioned into blocks of size $b$. If
$$
b \geq \frac{\Delta}{s} \Bigl( 1 + c \sqrt{\frac{\log s}{s}} \Bigr),
$$
for some constant $c$, then $G$ has a transversal which does not contain any $s$-stars, which can be found in randomized polynomial time.
\end{theorem}
\begin{proof}
We use a fractional hitting set which assigns weight $1/\binom{s}{r}$ to each $r$-star, and zero to every other subgraph, for some integer $r \leq s$. We will use Theorem~\ref{resample-main-thm2} with $K = 1$, assigning the constant vector 
$$
\vec \lambda = \alpha = \frac{s-r}{\Delta (1+r)^{2/r}},
$$
As any two $r$-stars $H, H'$ which both correspond to the same $s$-star, will overlap in their central vertex, we see that $\bowtie$ is null so $S_1 = 0$.

For any vertex $v$, there are at most $\binom{\Delta}{r}$ $r$-stars in which $v$ is the central vertex and at most $\Delta \binom{\Delta-1}{r-1}$ $r$-stars where it is a peripheral vertex. So the condition of Theorem~\ref{resample-main-thm2} becomes
\begin{equation}
\label{rt1}
b \alpha - b \frac{\Bigl( \mybinom{\Delta}{r} +\Delta \mybinom{\Delta-1}{r-1} \Bigr)}{\binom{s}{r}} \alpha^{r+1} \geq 1
\end{equation}

We estimate this as:
$$
b \alpha - b \frac{\Bigl(\mybinom{\Delta}{r} +\Delta \mybinom{\Delta-1}{r-1} \Bigr)}{\binom{s}{r}} \alpha^{r+1} \geq b \alpha - b \Bigl( (r+1) \Delta^r \frac{ (s-r)! }{s!} \Bigr)\alpha^{r+1} \geq b \alpha - \frac{b (r+1) \Delta^r \alpha^{r+1}}{(s-r)^r}
$$

Thus, substituting the value of $\alpha$, a sufficient condition to satisfy (\ref{rt1}) is given by
$$ 
b \geq (\Delta/s) \times \frac{ (r+1)^{1+2/r}}{r(1-r/s)}
$$

At this point we set $r = \lceil \sqrt{s \ln s} \rceil$. As $(r+1)^{1+2/r}/r$ is a decreasing function of $r$, we have:
$$
\frac{ (r+1)^{1+2/r}}{r (1-r/s) } \leq \frac{ (1 + \sqrt{s \ln s})^{1 + \frac{2}{\sqrt{s \ln s}}}}{\sqrt{s \ln s} (1 - \frac{1 + \sqrt{s \ln s}}{s}) } \leq 1 + O(\sqrt{\frac{\log s}{s}})
$$

This implies that the PRA converges under the criterion
$$
b \geq (\Delta/s)  \Bigl( 1 + c \sqrt{\frac{\log s}{s}} \Bigr)
$$
for some sufficiently large constant $c > 0$.

To implement a step of the PRA, one must search the graph for any $s$-star in the current candidate transversal; this can be done easily in polynomial time.
\end{proof}

Theorem~\ref{thm:avoidclique} improves on \cite{locally-sparse} in three distinct ways: it gives a better asymptotic estimate for $\gamma_s$; it is fully constructive; it finds a transversal omitting not only $s$-cliques but also $s$-stars.

\subsection{Bounds in terms of average block degree}
The \emph{maximum degree} $\Delta$ is a relatively crude statistic. Let us define $d$ to be the maximum \emph{average} degree of any class $V_i$ and get bounds in $d$ instead. (Formally, we take the average of the degree (in $G$) of all vertices in $V_i$, and then maximize this over all $i$). For some graphs $H$, the PRA gives a simple method of finding independent transversals avoiding $H$, where $b$ is a simple function of $d$.

We say that a graph $H$ is \emph{intersecting} if for all edges $f, f'$ of $H$ we have $f \cap f' \neq \emptyset$.  Note that an intersecting graph $H$ with $r$ edges is an $r$-star, unless $r = 3$ in which case $H$ can also be a triangle.
\begin{theorem}
\label{ind-traversal-thm}
Let $H$ be an intersecting graph with $r$ edges. Let $G$ be a graph whose vertex set is partitioned into blocks of size at least $b$, and let $d$ be the maximum average degree of any block. If $b \geq 4 d/r$, then $G$ has a transversal avoiding $H$ which can be found in randomized polynomial time.
\end{theorem}
\begin{proof}
For each block $i$, we define a variable $X_i$ which is the choice of which vertex in that block to place into the transversal. We give $X_i$ has the uniform distribution over vertices in that block.

We will apply Theorem~\ref{resample-main-thm2} with $K = 1$. Each copy of $H$ in $G$ corresponds to an atomic bad-event of $\mathcal B_1$. We define a fractional hitting-set $Q_1$ by setting $Q_1(\{u, v \}) = 1/r$ for each edge $f = (u, v) \in E$, and $Q_1$ is zero everywhere else.

Since the atomic bad events all involve exactly $r$ edges, $Q_1$ satisfies the conditions of fractional hitting-set. Furthermore, any pair of edges $f_1, f_2$ which are both part of a copy of $H$ overlap in at least one vertex, so $\bowtie$ is null and $S_1 = 0$. 

The vector $\lambda$ used for Theorem~\ref{resample-main-thm2} has all its entries equal to a scalar value $\alpha \geq 0$. Let $d_v$ denote the degree of vertex $v$. Then, in order to prove $\lambda_i \geq 1 + \sum_k G_i(Q_k, \lambda)$, we need to show
$$
b \alpha - \sum_{v \in V_i} d_v \alpha^2/r \geq 1,
$$

By definition of $d$, we have $\sum_{v \in V_i} d_v \leq b d$, so we need to show
$$
b \alpha - bd \alpha^2/r \geq 1
$$

When $b \geq 4 d / r$, this has a solution $\alpha$ with $0 < \alpha \leq \frac{r}{2 d}$. This shows that the PRA converges in time $O(n \alpha) \leq O(n r / 2 d)$. Note that $H$ must either be a triangle, or an $r$-star; either of these can be found in polynomial time in $G$.
\end{proof}

As shown in \cite{jin:transversals,yuster:transversals,szabo-tardos:extremal}, when $H = K_2$ this cannot be improved to $b \geq 2 d -1$. As shown in \cite{locally-sparse}, when $H$ is a triangle this cannot be improved to $b \geq (1-\epsilon) d$ for any constant $\epsilon > 0$.  

\section{Packet routing}
\label{sec:routing}
Consider a graph $G$ with $N$ packets, each of which has a specified simple path of length at most $D$ to reach its endpoint vertex (we refer to $D$ as the \emph{dilation}). In any timestep, a packet may wait at its current position, or move along the next edge on its path. Our goal is to find a schedule of smallest makespan in which, in any given timestep, an edge carries at most a single packet.

We begin by reviewing the basic strategy of \cite{lmr}, and its improvements by \cite{schei:thesis} and \cite{peis-wiese}. We recommend consulting \cite{schei:thesis}, which has a very detailed explanation of this problem as well as many more variants than we cover here. We note that \cite{peis-wiese} studied a more general version of the packet-routing problem, so their choice of parameters was not (and could not be) optimized.

We define the \emph{congestion} $C$ to be the maximum, over all edges, of the number of packets scheduled to traverse that edge. It is clear that $D$ and $C$ are both lower bounds for the makespan, and \cite{lmr} has shown that in fact a schedule of makespan $O(C+D)$ is possible. The work of \cite{schei:thesis} provided an explicit constant bound of $39 (C+D)$, as well as describing an algorithm to find such a schedule. This was improved to $23.4 (C+D)$ in \cite{peis-wiese} as will be described below.

While the final schedule only allows one packet to cross an edge at a time, we will relax this constraint during our construction. We consider ``infeasible'' schedules, in which arbitrarily many packets pass through each edge at each timestep. We define an \emph{interval} to be a consecutive set of times in our schedule, and the \emph{congestion} of an edge in a given interval to be the number of packets crossing that edge. If we are referring to intervals of length $i$, then we define a \emph{frame} to be an interval which starts at an integer multiple of $i$.

From our original graph, one can easily form an (infeasible) schedule with delay $D$ and overall congestion $C$. Initially, this congestion may ``bunch up'' in time, that is, certain edges may have very high congestion in some timesteps and very low congestion in others. Our construction  will ``even out'' the schedule, bounding the congestion on successively smaller intervals.

Ideally, this process would eventually finish with each individual timestep (i.e. interval of length 1) having congestion roughly $C/D$. In this case, the infeasible schedule could be turned into a feasible schedule, by simply expanding each timestep into $C/D$ separate timesteps.

Peis \& Wiese \cite{peis-wiese} improved the bound on the makespan to $23 (C+D)$ by controlling the congestion on intervals of length $2$ (instead of length $1$). Given our infeasible schedule, we can view each interval of length 2 as defining a new subproblem with dilation $2$ and congestion $C$. We quote their result:
\begin{proposition}[\cite{peis-wiese}]
\label{peisprop}
If $D = 2$, there is a schedule of makespan $C + 1$ that can be found in polynomial time.
\end{proposition}

\subsection{Using the LLL to find a schedule}
As a starting point, our construction is based on \cite{schei:thesis} with some optimized parameters. We add random delays to each packet, and then allow the packet to move through each of its edges in turn without hesitation. The LLL is used to ensure that the congestion does not get too large on any interval.

\begin{lemma}
\label{lmrlem1}
Suppose there is a schedule $S$ of length $L$ such that every interval of length $i$ has congestion at most $C$. For positive integers $m, C', i'$ with $i' \leq i/2$, suppose that
$$
e \times \Pr(\text{Binomial} (C, \frac{i'}{i - i'}) > C') \times (C m i^2 + 1) < 1
$$

Then there is a schedule $S'$ of length $L' = L(1 + 1/m) + i$, in which every interval of length $i'$ has congestion at most $C'$. Furthermore, $S'$ can be constructed in expected polynomial time.
\end{lemma}

\begin{proof}
We break the schedule $S$ into frames of length $F = m i$, and refine each separately. Within each frame, we add a random delay to each packet separately. The delays are uniformly distributed in the range $\{0, \dots, i - i' - 1 \}$ and are independent. (We refer to this as \emph{adding a random delay in the range $i - i'$} to each packet)

Let us fix an $F$-frame for the moment. Each edge $f$ and $i'$-interval $I$ has a bad event $\mathcal B_{f,I}$ that the congestion exceeds $C'$. Each $f, I$ has at most $C$ possible packets that could traverse it, and each does so with probability at most $p = \frac{i'}{i - i'}$. Hence the probability of $\mathcal B_{f,I}$ is at most the probability that a binomial random variable with $C$ trials and probability $p$ exceeds $C'$.

If a packet $x$ was originally scheduled to cross some edge $f'$ at time $s \leq mi$ in the schedule $S$, then in the schedule $S'$ it potentially affects $f'$ within intervals $\{s - i' + 1, \dots, s \}, \dots, \{ s+ (i-i'), \dots, s+(i - i') + i' - 1 \}$, a total of $i$ intervals. Thus, overall $x$ can affect at most $m i^2$ other events $\mathcal B_{f', I'}$. Since there are at most $C$ packets which could affect $f, I$, this implies that each $\mathcal B_{f,I}$ affects at most $d = C m i^2$ other bad-events.

By the LLL, the condition in the hypothesis guarantees that there is a positive probability that the delays avoid all bad events. In this case, we refine each frame of $S$ to obtain a new schedule $S'$ as desired. We can use the MT algorithm to actually find schedule $S'$ in polynomial time.

So far, this ensures that \emph{within each frame}, the congestion within any interval of length $i'$ is at most $C'$. In the refined schedule $S'$ there may be intervals that cross frames. To ensure that these do not pose any problems, we insert a delay of length $i'$ between successive frames, during which no packets move at all. This step means that the schedule $S'$ may have length up to $L (1 + 1/m) + i$. 
\end{proof}

 Lemma~\ref{lmrlem1} allows us to transform the original problem instance into one where $C, D$ are small finite values, with a negligible cost to the approximation ratio. For simplicity here, we focus on the case in which $C, D$ are very large and so certain rounding effects can be disregarded.  
\begin{lemma}
\label{lmrlem2}
Assume $C+D \geq 2^{896}$. There is a schedule of length at most $1.004 (C+D)$ and in which the congestion on any interval of length $2^{24}$ is at most $17040600$. Furthermore, this schedule can be produced in randomized polynomial time.
\end{lemma}
\begin{proof}
We provide a sketch here; see \cite{schei:thesis} for a much more thorough explanation of this process.

Define the sequence $a_k$ recursively as $a_0 = 256$ and $a_{k+1} = 2^{a_k}$. There is a unique $k$ with $a_k^{3.5} \leq C+D < a_{k+1}^{3.5}$. By a slight variant on Lemma~\ref{lmrlem1}, one can add delays to obtain a schedule of length $C+D$, in which the congestion on any interval of length $i' = a_k^3$ is at most $C' = i' (1 + 4/a_k)$.

At this point, we repeatedly apply Lemma~\ref{lmrlem1} with $i = a_j, i' = a_{j+1}$, for $j = k-1, \dots 0$. At each step, this increases the length of the resulting schedule from $L_j$ to $L_j (1 + 1/a_{j+1}) + a_j$, and increases the congestion on the relevant interval from $i (1+4/a_k) $ to
$$
i (1 + 4/a_k) \prod_{j=0}^{k-1} (1 +  4/a_{j}) (\frac{1}{1 - (a_j/a_{j+1})^3})
$$
(We use the Chernoff bound to estimate the binomial tail in Lemma~\ref{lmrlem1}.)

For $C+D \geq a_k^{3.5}$, a simple calculation shows that the schedule length increases from $C+D$ (after the original refinement) to at most $1.004 (C+D)$. In the final step of this analysis, we are bounding the congestion of intervals of length $a_0^3 = 2^{24}$, and the congestion on such an interval is at most $17040600$. 
\end{proof}

\begin{lemma}
\label{schedlem1}
If $C + D \geq 2^{896}$, then there is a feasible schedule of length at most $10.92(C+D)$ which can be constructed in randomized polynomial time.
\end{lemma}
\begin{proof}
By Lemma~\ref{lmrlem2}, we form a schedule $S_1$, of length $L_1 \leq 1.004 (C+D)$, in which each interval of length $2^{24}$ has congestion at most $17040600$.

Now apply Lemma~\ref{lmrlem1} to $S_1$, with $m = 64, i' = 1024, C' = 1385$ to obtain a schedule $S_2$, of length $L_2 \leq 1.0157 L_1 + 2^{24}$, in which each interval of length $1024$ has congestion at most $1385$.

Now apply Lemma~\ref{lmrlem1} to $S_2$ with $m = 64, i' = 2, C' = 20$, to obtain a schedule $S_3$ of length $L_3 \leq 1.0157 L_2 + 1024$, in which each frame of length 2 has congestion at most $20$.

Now apply Proposition~\ref{peisprop} to $S_3$, expanding each $2$-frame to a \emph{feasible} schedule of length $21$. The total length of the resulting schedule is at most $\frac{21}{2} L_3 \leq 10.92 (C+D)$.
\end{proof}

\subsection{The PRA applied to packet routing}
The schedule modification in Lemma~\ref{lmrlem1} essentially comes down to an assignment-packing problem: within each frame, we assign a delay to each packet, and a bad event corresponds to an edge receiving an excessive congestion in some time interval. We thus modify Lemma~\ref{lmrlem1} to use the PRA instead of the LLL.

\begin{proposition}
  \label{noncons1}
  Suppose there is a schedule $S$ of length $L$ such that every interval of length $i$ has congestion at most $C$. Let $m, C', d, i'$ be positive integers with $i' < i$ and $d \leq C'$, and let $\alpha \in [0,1]$ be a real number.
  
Define 
$$
 p = \frac{(C i' \alpha)^d} {d! \binom{C'+1}{d}}
$$

Suppose that $p < 1$ and 
$$
(i - i') \alpha - \frac{ m i^2 d p}{C (1-p)} \geq 1
$$

Then there is a schedule $S'$ of length $L' = L(1 + 1/m) + i$, in which every interval of length $i'$ has congestion at most $C'$. Furthermore, such a schedule can be found in polynomial time.
\end{proposition}
\begin{proof}
Suppose we add delays in the range $\{0, \dots, i - i' - 1 \}$ uniformly to each packet within each frame of length $F = m i$.  In this case, we have a variable corresponding to each packet $x$, and for each delay $\delta$ we assign $\lambda_{x,{\delta}} = \alpha$. For each edge $f$ and $i'$-interval $I$, we have a complex bad event $\mathcal B_{f, I}$ that the congestion in the interval exceeds $C'$. Each such bad-event uses the fractional hitting-set $Q_{f,I}$ with parameter $d$ as described in Theorem~\ref{devthm}.

For a given $f, I$,  we must compute $\mu$, which is the total contribution of $\lambda$ summed over all packets/delays which could contribute to the congestion of that edge-interval. There are at most $C$ packets which could be scheduled to pass through the given edge, and there are $i'$ possible delays which affect $f,I$. So, in all, the total contribution is  $\mu \leq C i' \alpha$. The bad event is that this exceeds $C'$, so $t = C' + 1$. By Theorem~\ref{devthm}, this gives $S_{f,I} = S(\mathcal B_{f,I}, Q_{f,I}, \lambda) \leq \frac{\mu^d} {d! \binom{C' + 1}{d}} \leq p$.

Next, consider some packet $x$; we wish to compute $\mu_x$, which is the total contribution to the bad-event $\mathcal B_{f,I}$ summed over all possible delays to packet $x$. There are at most $i'$ delays which can cause $x$ to transit $f$ within $I$, hence $\mu_x \leq i' \alpha$ if packet $x$ could cross edge $f$ in interval $I$. So $G_x(Q_{f,I}, \lambda) \leq \frac{i' \alpha}{\mu} d \frac{\mu^d} {d! \binom{C' + 1}{d}} \leq \frac{d p}{C}$. 

Each packet $x$ affects up to $m i$ edges within the frame; if a packet $x$ was originally scheduled to cross an edge $f$ at time $s \leq m i$ in the schedule $S$, then in the schedule $S'$ it potentially affects $f$ within the $i$ intervals $\{s-i'+1, \dots, s \}, \dots, \{ s+(i-i'), \dots, s+(i-i')+ i'-1 \}$. So summing $G_x(Q_{f, I}, \lambda)$ over all $f, I$ affected by packet $x$ yields 
$$
\sum_{f,I} G_x(Q_{f,I}, \lambda ) \leq \frac{ m i^2 d p}{C}
$$

In order to apply Theorem~\ref{resample-main-thm2} each packet $x$ must satisfy the constraint
\begin{equation}
\label{hh1}
\lambda_x \geq  1 + \sum_{f,I} \frac{G_x(Q_k, \lambda)}{1 - S_{f,I}}.
\end{equation}

Each packet $x$ has $\lambda_x = (i - i') \alpha$ and $S_{f,I} \leq p$ and $\sum_{f,I} G_x(Q_{f,I}, \lambda) \leq m i^2 d p / C$, so (\ref{hh1}) becomes
$$
(i - i') \alpha \geq 1 + \frac{m i^2 d p}{(1-p) C}
$$

This is precisely the constraint specified in the hypothesis. The expected number of resamplings is $\sum_{x, \delta} \lambda_{x, \delta} \leq N D \alpha$, which is polynomially bounded.
\end{proof}

We can use this to improve various steps in the construction.
\begin{proposition}
\label{noncons3}
Suppose $C+D \geq 2^{896}$. Then there is a schedule of length $\leq 1.00652 (C+D)$, in which every interval of length $1024$ has congestion at most $1320$, which can be constructed in randomized polynomial time.
\end{proposition}
\begin{proof}
By Lemma~\ref{lmrlem2}, form a schedule $S_1$, of length $L_1 \leq 1.004 (C+D)$, in which each interval of length $2^{24}$ has congestion at most $17040600$. 
Apply Proposition~\ref{noncons1} with $\alpha = 5.98328\times 10^{-8}, C' = 1320, d = 270, m = 400$ to obtain a schedule $S_2$ of length $L_2 \leq 1.0025 L_1 + 2^{24} \leq 1.00652 (C+D)$, in which each interval of length $1024$ has congestion at most $1320$.
\end{proof}

\begin{theorem}
\label{noncons4a}
Suppose $C+D \geq 2^{896}$. Then there is a schedule of length at most $8.61 (C+D)$ which can be constructed in randomized polynomial time.
\end{theorem}
\begin{proof}
By Proposition~\ref{noncons3}, there is a schedule $S_1$ of length at most $L_1 = 1.00652 (C+D)$ in which each interval of length $1024$ has congestion at most $1320$.

Now apply Lemma~\ref{lmrlem2} with $i = 1024, C = 1320, i' = 2, m = 100, C' = 16, d = 12, \alpha = 0.00107911$ to obtain a schedule $S_2$ of length $L_2 \leq 1.01 L_1 + 1024$, in which each interval of length $2$ has congestion at most $16$.

Now apply Proposition~\ref{peisprop} to $S_2$, expanding each $2$-frame to a \emph{feasible} schedule of length $17$. The total length of the resulting schedule is at most $\frac{17}{2} L_2 \leq 8.61 (C+D)$.
\end{proof}

\subsection{Better scheduling of the final 2-frame}
\label{better4}
Let us examine more closely the penultimate stage in the proof of Theorem~\ref{noncons4a}, in which the schedule gets divided into $2$-frames where the congestion of each edge is bounded by some parameter $C'$. For a given edge $f$ and time $t$, we define $c_t(f)$ to be the number of packets scheduled to cross $f$ at time $t$; for a given packet $x$ and time $t$, we define $e_t(x)$ to be the identity of the packet edge at time $t$ (possibly there is no edge, in which case $e_t(x) = \emptyset$).

For a given value of $t$, it is relatively likely that $c_t(f) + c_{t+1}(f)$ or $c_{t+2}(f) + c_{t+3}(f)$ are much larger than their mean.  However, it is unlikely that \emph{both} these bad events happen simultaneously. To take advantage of this, we construct a schedule in which we insert an ``overflow'' time between the 2-frames to handle the situation in which \emph{either} $c_t(f) + c_{t+1}(f)$ \emph{or} $c_{t+2}(f) + c_{t+3}(f)$ is too large.  Our goal will be to modify either of the intervals $\{t, t +1 \}$ or $\{t+2, t+3 \}$ to ensure that both have congestion at most some parameter $T$. 

For a given 2-frame $I$, we add two overflow time slots, before and after $I$, to schedule the excess packets. If an edge $f$ has more than $T$ transits scheduled during the interval $I$, then we can fix this by either finding some packet $x$ with $e_1(x) = f$ and shifting it into the earlier overflow time, or by finding some packet $x$ with $e_2(x) = f$, and shifting it into the later overflow time. See Figure~\ref{packetfig1}.

\vspace{0.6in}

\begin{figure}[h]
\begin{center}
\includegraphics[trim = 2.5cm 21.5cm 3.5cm 4cm,scale=0.5,angle = 0]{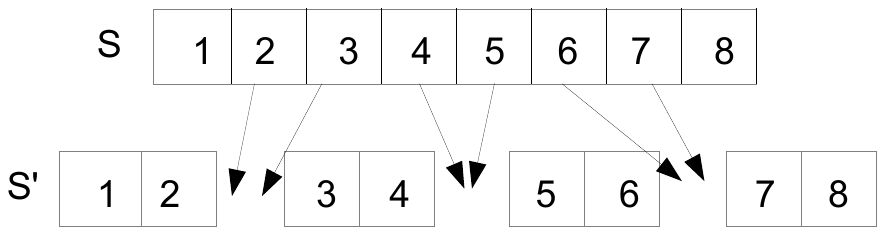}
\vspace{-0.2in}
\caption{The packets in the original schedule $S$ are shifted into overflow times in the schedule $S'$.
\label{packetfig1}}
\end{center}
\end{figure}

We need to be careful to account for how often a given edge $f$ appears as $e_t(x)$ or $e_{t+1}(x)$. For example, if there are no (remaining) packets with $e_{t+1}(x) = f$, then we are only allowed to shift $f$ into the earlier overflow, and similarly if there are no remaining packets with $e_{t}(x) = f$. Keeping this constraint in mind, we seek to equalize as far as possible the distribution of edges into earlier and later overflows. We do so as follows:

\begin{algorithmic}[1]
\FOR{each edge $f$ and each odd integer $t$}
\WHILE{$c_t(f) + c_{t+1}(f) > T$}
\STATE if $c_t(f) = 0$, then shift one packet into the later overflow time.
\STATE else if $c_{t+1}(f) = 0$, then shift one packet into the earlier overflow time.
\STATE else if $c_t(f) + c_{t+1}(f) = \text{odd}$, then shift one packet into the earlier overflow time.
\STATE else if $c_t(f) + c_{t+1}(f) = \text{even}$, then shift one packet into the later overflow time.
\ENDWHILE
\ENDFOR
\end{algorithmic}

For any odd integer $t$ and edge $f$, let $c'$ denote the congestions at the end of this overflow-shifting process, so that $c'_t(f) + c'_{t+1}(f) \leq T$. The number of packets shifted into the earlier (respectively later) overflow time can be viewed as a function of the original values $c_{t}(f), c_{t+1}(f)$. We denote these ``overflow'' functions by $\text{OF}^{-} (c_t, c_{t+1}; T)$ and $\text{OF}^{+} (c_{t},c_{t+1}; T)$ respectively. Specifically we get the following condition:
\begin{proposition}
\label{4frameprop}
Let $S$ be a schedule of length $L$, and let $c_t(f)$ for $t = 1, \dots, L$ denote the number of times $f$ is scheduled as the $t^{\text{th}}$ edge of a packet. Suppose that $T \geq T' \geq 1$, and suppose that for all edges $f$ and all odd integers $t$ the schedule $S$ satisfies the constraint
$$
\text{OF}^{+}(c_t(f), c_{t+1}(f);T) + \text{OF}^{-} (c_{t+2} (f),c_{t+3} (f); T) \leq T'
$$
Then there is a feasible schedule $S'$ of makespan $(L + 1)/2 \times (T + T' + 2) - 2$, which can be constructed in polynomial time.  (Note that for $t < 1$ and $t > L$, we define $c_t(f) = 0$.)
\end{proposition}
\begin{proof}
After the modification, each 2-frame has congestion at most $T$, while each overflow time has congestion at most $T'$. Each overflow time has delay at most 2, since for any packet $x$, there may be at most two edges scheduled into that overflow time, namely the edge that had been originally marked as the second edge of the earlier 2-frame, and the edge that had been originally marked as the first edge of the latter 2-frame. Hence each 2-frame can be scheduled in time $T + 1$ and each overflow can be scheduled in time $T' + 1$. Also note that the final and initial overflow times have delay $1$, so they can be scheduled in time $T'$. 

Let us first suppose that $L$ is even. There are $L/2$ 2-frames in the original schedule and so there are $L/2 + 1$ overflow periods. Hence the total cost is at most $\tfrac{L}{2} (T + 1)  + (\tfrac{L}{2} - 1) (T'+1) + 2 T' = \tfrac{L}{2} (2 + T + T') + T' - 1$. The condition $T \geq T'$
implies that this is at most $(L + 1)/2 \times (T + T' + 2) - 2$

If $L$ is odd, then we can merge the final overflow time into the final ordinary time; in this final $2$-frame, each edge has congestion at most $T + T'$. Thus, overall the cost is at most $\tfrac{L-1}{2} (T + 1) + T' + (\tfrac{L-1}{2} - 1) (T' + 1) + (T + T' + 1) = (L + 1)/2 \times (T + T' + 2) - 2$.
\end{proof}

The conditions required by Proposition~\ref{4frameprop} are local, in the sense that any violation is any event which affects an individual edge and a 4-interval which starts at an odd time $t$.  We refer to the conditions required by this Proposition as the \emph{4-interval-conditions (abbreviated 4IC)}; these conditions can be viewed as either pertaining to an entire schedule, or to individual 4-interval.

We will  use the PRA to find a schedule satisfying the conditions of Proposition~\ref{4frameprop}.

\begin{proposition}
\label{noncons2}
Let $m = 36 , C = 1320 , i = 1024,  T = 6, T' = 5$.

Suppose there is a schedule $S$ of length $L$ such that every interval of length $i$ has congestion at most $C$. There is a schedule of length $L' \leq (1 + 1/m) L + i$, which satisfies the 4IC with respect to $T, T'$. This schedule can be produced in polynomial time.
\end{proposition}
\begin{proof}
Our plan is to break the schedule into frames of size $F = m i$; within each packet and frame we add a random delay in the range $i - 4$. Let us fix a frame for the moment.

For each edge $f$, and 4-interval $I$ starting at time $s$ where $s$ is an odd integer, we have a complex bad event $\mathcal B_{f,I}$ that $$
\text{OF}^+ (c_s(f), c_{s+1}(f); T) + \text{OF}^{-}(c_{s+2}(f), c_{s+3}(f); T) > T'
$$

For this edge/interval $f, I$, and any packet $x$ with delay $t$, we say that $\langle x, t \rangle$ has \emph{type $j$}, if that packet-delay assignment would cause the given packet $x$ to land at position $s + j$ within the bad event, for $j = 0,\dots,3$. If that assignment $x, s$ does not contribute to $\mathcal B_{f, I}$, then $\langle x, s \rangle$ has no type. Note that each $f, I$ has at most $C$ packet-delay combinations of each type.

For a bad event $\mathcal B_{f,I}$ and a fractional hitting-set $Q_{f,I}$, we define the quantities $\Phi$ and $\Phi_j$, for $j = 0, 1, 2, 3$, as 
$$
\Phi_j = \max_{f,I} \max_{\substack{\text{$\langle x, t \rangle$ has type}\\ \text{$j$ for $f, I$}}} \sum_{Y \ni \langle x, t \rangle} Q_{f,I}(Y) \lambda^Y, \qquad \Phi = \max_{f,I} \sum_{Y} Q_{f,I}(Y) \lambda^Y 
$$

 We will apply Theorem~\ref{resample-main-thm2} using a separate complex-bad event for each $f, I$, and a separate variable for each packet (the value of a variable is the chosen delay), and the vector $\vec \lambda_{x,t} = \alpha = 0.001051$. For each such label $f,I$ we have $S_{f,I} = S(\mathcal B_{f,I}, Q_{f,I}, \lambda) \leq \sum_{Y} Q_{f,I}(Y) \lambda^Y \leq \Phi$. 

For any packet $x$ and delay $t$ and $j = 0, \dots, 3$, there are at most $m i/2$ pairs $f,I$ for which packet $x,t$ has type $j$. Each such $f,I$ has $G_{x,t}(Q_{f,I}, \lambda) \leq \Phi_j$. Summing over $f, I$ and the $i$ choices for the delay $t$ gives
$$
\sum_{f, I} G_x (Q_{f,I}, \lambda) \leq \frac{m i^2}{2} (\Phi_0 + \Phi_1 + \Phi_2 + \Phi_3)
$$

By Theorem~\ref{resample-main-thm2}, each packet $x$ must satisfy the constraint
$$
\lambda_x \geq 1 + \sum_{f,I} \frac{G_x(Q_{f,I}, \lambda)}{1 - S_{f,I}}
$$
We have $\lambda_x = (i-4) \alpha$, and so it suffices to satisfy the condition
\begin{equation}
\label{ss1}
(i - 4) \alpha - \frac{m i^2}{2} \frac{\Phi_0 + \Phi_1 + \Phi_2 + \Phi_3}{1 - \Phi} \geq 1
\end{equation}
in order to find acceptable delays. Such delays lead to a schedule of length $L' \leq (1 + 1/m) L + i$, which satisfies the 4IC with $T, T'$.

Thus, we have reduced our problem to constructing fractional hitting-sets $Q_{f,I}$ which have a sufficiently small value for $\frac{\Phi_0 + \Phi_1 + \Phi_2 + \Phi_3}{1 - \Phi}$. Although we have stated the proposition for a particular choice of parameters, we will walk through the algorithm we use to construct it next.

The bad event depends $\mathcal B_{f,I}$ is determined by the number of variables of each type assigned to edge $f$ on interval $I$. There are at most $C$ such variables of each type; to simplify the notation, we suppose there are exactly $C$. The fractional hitting-set $Q_{f,I}$ assigns weights to any subset of the $4 C$ variables involved in the bad event; we can write such a subset as $Y = Y_1 \cup Y_2 \cup Y_3 \cup Y_4$, where the packet/delays in $Y_j$ all have type $j$.

We will make $Q_{f,I}$ symmetric, in the sense that for any such $Y = Y_0 \cup Y_1 \cup Y_2 \cup Y_3$, the value of $Q_{f,I}(Y)$ depends solely on the cardinalities $|Y_0|, |Y_1|, |Y_2|, |Y_3|$. Thus, we define
$$
Q_{f,I}(Y_0 \cup Y_1 \cup Y_2 \cup Y_3) = b(|Y_0|, |Y_1|, |Y_2|, |Y_3|)
$$
where $b: [C]^4 \rightarrow [0,1]$ is a function which we will determine. Let us define $\hat \Phi_0$, which serve as an upper bound on $\Phi_0$, by 
$$
\hat \Phi_0 = \sum_{y_0, y_1, y_2, y_3} \tbinom{C - 1}{y_0 - 1} \tbinom{C}{y_1} \tbinom{C}{y_2} \tbinom{C}{y_3} b(y_0 , y_1, y_2, y_3) \alpha^{y_0 + y_1 + y_2 + y_3 }
$$
and similarly for $\hat \Phi_1, \hat \Phi_2, \hat \Phi_3, \hat \Phi$. Here $y_0, y_1, y_2, y_3$ denote the possible cardinalities of $Y_0, Y_1, Y_2, Y_3$ respectively. (The reason for the term $\binom{C-1}{y_0 - 1}$ here, as opposed to $\binom{C}{y_0}$, is that in computing $\Phi_0$, we have fixed the presence of a single packet/delay $\langle x, t \rangle$ with type $0$ for the given $f, I$. Thus, there are only $\binom{C-1}{y_0-1}$ choices for the \emph{additional} type-$0$ packets involved in $f,I$.)

We say a tuple $(k_0, k_1, k_2, k_3)$ is \emph{bad} if it satisfies 
$$
\text{OF}^+ (k_0, k_1; T) + \text{OF}^{-}(k_2, k_3; T) > T'
$$
and we say it is \emph{minimal bad} if it is bad, but no other $(k_0', k_1', k_2', k_3')$ strictly smaller than it is bad.

The fractional hitting-set must satisfy $\sum_{Y \subseteq A} Q(Y) \geq 1$ for any $A \in \mathcal B_{f,I}$. By symmetry, this means that if $k_0, k_1, k_2, k_3$ are minimal bad, then we require
\begin{equation}
\label{r1}
\sum_{y_0, y_1, y_2, y_3} \tbinom{k_0}{y_0} \tbinom{k_1}{y_1} \tbinom{k_2}{y_2} \tbinom{k_3}{y_3} b(y_0, y_1, y_2, y_3) \geq 1
\end{equation}

We set $b(y_0, y_1, y_2, y_3) = 0$ unless there is some minimal bad tuple $(k_0, k_1, k_2, k_3)$ with $k_0 \geq y_0, \dots, k_3 \geq y_3$. 

We are trying to satisfy $(\hat \Phi_0 + \hat \Phi_1 + \hat \Phi_2 + \hat \Phi_3)/(1-\hat \Phi) \leq t$ for some target value $t$. When $t$ is fixed, this is equivalently to minimizing $\hat \Phi_0 + \hat \Phi_1 + \hat \Phi_2 + \hat \Phi_3 + t \hat \Phi$. If we view the collection of all values $b(y_0, y_1, y_2, y_3)$ as linear unknowns, then we can view both the objective function and the constraints as linear. Hence this defines an LP, which we can solve using standard algorithms. We can then optimize $t$ by binary search.

For $T = 6, T' = 5$, the resulting linear program has $12000$ variables and $259$ constraints. This is too large to write explicitly, but we wrote computer code to generate and solves it. The resulting hitting-set achieves a bound of 
$$
\frac{\hat \Phi_0 + \hat \Phi_1 + \hat \Phi_2 + \hat \Phi_3}{1-\hat \Phi} \leq 3.495 \times 10^{-9}
$$
and this satisfies (\ref{ss1}). (We recommend that any reader who wishes to recover it should construct the linear program of (\ref{r1}) and solve it for themselves.)

To show that the number of resamplings is polynomially bounded, note that we are treating each $F$-frame separately. Thus, all of the quantities $\sum_x \lambda_x$ etc. for the PRA are functions of parameters $m, C, i, T, T'$: in particular, they do not depend on the overall problem size. So, the expected number of resamplings within each frame is some constant, and so the overall number of resamplings is $\text{poly}(N)$.
\end{proof}

We now apply this construction to replace the two final steps in the construction of Section~\ref{better4}.
\begin{theorem}
\label{final-thm}
There is a feasible schedule of makespan at most $6.73(C+D)$, which can be constructed in randomized polynomial time.
\end{theorem}
\begin{proof}
We give the full proof in Appendix~\ref{full-proof-app} (which has numerous cases and calculations). Here, we focus on the most interesting case, where $C+D \geq 2^{896}$.

 By Proposition~\ref{noncons3}, we obtain a schedule $S_1$ of length $L_1 \leq 1.00652 (C+D)$, in which each interval of length $1024$ has congestion at most $1320$.

Now apply Proposition~\ref{noncons2}. This gives a schedule $S_2$ of length $L_2 \leq 1.02779 L_1 + 1024$ satisfying the 4IC with $T = 6, T' = 5$. By Proposition~\ref{4frameprop}, this yields a schedule whose makespan is $6.5 L_2 + 4.5 \leq 6.73 (C+D)$.
\end{proof}

\section{Acknowledgments}
We thank Tom Leighton and Satish Rao for valuable discussions long ago, which served as the foundation for this work;
but for their declining, they would be coauthors of this paper. We are thankful to Noga Alon, Venkatesan Guruswami, Bernhard Haeupler, 
Penny Haxell, and Jeff Kahn for helpful discussions, as well as to the STOC 2013, FOCS 2013, and journal referees for their valuable comments.

Aravind Srinivasan dedicates this work to the memory of his late grandmother Mrs.\ V.\ Chellammal (a.k.a.\ Rajalakshmi): memories of her
remain a great inspiration.

\bibliographystyle{plain}
\bibliography{lll-largesupp}

\appendix

\section{Comparison with the MT algorithm}
\label{compare-mt-sec}
To compare our results with the MT algorithm, let us consider the class of assignment-packing problems where $c_k = R$ for all $k$. We have seen (Proposition~\ref{csp-thm3}) that the PRA converges in polynomial time if the bad-events are defined with RHS vector $b_k = R'$; here $R'$ is a function of $R$ and of the maximum $\ell_1$-norm $D$ of the constraint matrix. Crucially, $R'$ is scale-free: it does \emph{not} depend on the number of variables $n$ or number of constraints $m$.  By contrast, we show that MT cannot achieve in sub-exponential time \emph{any} value $b_k$ which depends solely on $R, D$ or the number of non-zeroes in each column $D'$. This holds even if the entries of $A$ are in the set $\{0, 1 \}$.

There are many possible parametrizations of the LLL and the MT algorithm for this problem, including strategies based on iterated applications. We consider only the simplest of these, which has a separate bad-event for each row of the constraint matrix.

\begin{proposition}
Let $R, R'$ be fixed real numbers with $1 \leq R \leq R'$. For $m$ sufficiently large there is an assignment-packing problem with the following properties:
\begin{enumerate}
\item The system has entries in $\{0, 1 \}$,  has $m$ constraints, and has $n = \Theta(R m)$ variables.
\item There is a fractional solution $z$ achieving RHS value $c_k = R$.
\item For each $i \in [n]$ and each possible assignment $j \in F_i$, there is exactly one row $k$ with $a_{k,i,j} > 0$. (So the matrix has maximum $\ell_0$ norm of $D' = 1$.)
\item Suppose we run the PRA, using resampling probabilities given by $p_{ij} = z_{ij}$, and with RHS values given by $b_k = R + C \sqrt{R}$ for a universal constant $C$. Then the PRA (with an appropriate  fractional hitting-set) terminates in expected polynomial time.
\item Suppose we run the MT algorithm, using resampling probabilities given by $p_{ij} = z_{ij}$ and with RHS values $b_k$ given by  $b_k = R'$. Then the probability that this MT algorithm terminates after $2^{\phi m}$ steps is at most $2^{-\phi' m}$, where $\phi, \phi' > 0$ are parameters depending solely on $R, R'$.
\end{enumerate}
\end{proposition}

\begin{proof}
For $i \in [n]$, the variable $X_i$  has domain $F_i = \{1, \dots, m \}$. To construct the constraint matrix $A$, we select for each $i \in [n]$ a permutation $\pi_i \in S_m$ independently and uniformly at random. For $k \in [m]$ we then set $a_{k,i,j} = [ \pi_i(k) = j ]$. Thus, all the entries of the constraint matrix are either $0$ or $1$. Also, observe that for any $i,j$, the only value of $k$ with $a_{k,i,j} > 0$ is given by $k = \pi_i^{-1}(j)$.

This system has a fractional solution of $z_{i,j} = 1/m$ for all $i, j$. For any value of $k$ this gives
$$
\sum_{i,j} a_{k,i,j} z_{i,j} = \sum_{i,j} [\pi_i(k) = j] / m = n/m
$$
For $n = \lfloor R m \rfloor$, this satisfies the constraints fractionally with RHS vector $c_k = R$.

The convergence of the PRA with this fractional vector $z$ follows from Proposition~\ref{csp-thm3}.

Finally, we will show that the MT algorithm requires a long time to terminate. To begin, we will show that for any fixed vector $x_1, \dots, x_n \in [m]^n$, the probability (over the random choice of $A$) that $A_k x \leq R'$ for all $k  \in [m]$, is at most $e^{-\Omega(m)}$. To show this, we view the vector of counts $A_1 x, \dots, A_m x$ as what is known as a \emph{competing ball-and-urns problem}; there is an urn corresponding to each $k \in [m]$, there is a ball corresponding to each $i \in [n]$, and the value of $A_k x$ is the number of balls placed into urn $k$. We place ball $i$ into urn $k$ iff $\pi_i(k) = x_i$ --- in other words, the placement of each ball is independently chosen among the $k$ urns.

Consider some fixed value of $k$. The value of $A_k x$ is a binomial random variable, with number of trials $n$ and success probability $1/m$, and so
{\allowdisplaybreaks
\begin{align*}
\Pr(A_k x > R') &\geq \binom{n}{R'+1} (1/m)^{R'+1} (1 - 1/m)^{n-R'-1} \geq \frac{(n-R')^{R'+1} (1-1/m)^{n-R'-1} }{ (R'+1)! m^{R'+1} } \\
&\geq \frac{(R m - 1 -R')^{R'+1} (1-1/m)^{R m-R'-1} }{ (R'+1)! m^{R'+1} } \qquad \text{as $R m - 1 \leq n \leq R m$} \\
&\geq \Omega(1) \qquad \text{for fixed $R, R'$.}
\end{align*}
}

As shown in \cite{dubhashi-ranjan}, the events $A_1 x \leq R', \dots, A_m x \leq R'$ are negatively correlated. Thus,
$$
\Pr( \bigwedge_{k=1}^m A_k \leq R' ) \leq \prod_{k=1}^m \Pr(A_k \leq R') \leq \prod_{k=1}^m \Bigl( 1 - (1 - \Omega(1)) \Bigr) \leq e^{-\Omega(m)}
$$

Every constraint in this CSP depends on every variable. Thus, whenever the MT resamples a bad-event, it  resamples all variables. So after the MT performs $T$ resamplings,  the current value of the variables $X_1, \dots, X_n$ is simply the $T^{\text{th}}$ row of the resampling table. Consequently, a necessary condition for the MT algorithm to terminate after $T$ steps is that one of the first $T$ rows of the resampling table satisfies all the constraints. Since any individual row satisfies all the constraints with probability $e^{-\Omega(m)}$, the probability that MT terminates after $T$ rounds is at most $T e^{-\Omega(m)}$. This is $e^{-\Omega(m)}$ unless $T \geq e^{\Omega(m)}$.
\end{proof}

\section{Some results on the MT distribution}
\label{mt-dist-appendix}
If Theorem~\ref{resample-main-thm} holds, then we know that there exists a configuration which avoids all bad events. We may wish to learn more about such configurations, other than that they exist. One useful tool is the \emph{MT-distribution}, which is the distribution on the variables $X_1, \dots, X_n$ at the termination of the PRA. We write $\pmt(E)$ to mean the probability of event $E$ in this space. (Note that, in order for this to be well-defined, we must choose a fixed rule for which bad-event to resample if there are multiple candidates; the bounds we derive hold for any such rule.) This probability space has been analyzed for the MT algorithm in \cite{haeupler-saha-srinivasan, harris2017algorithmic, harris2016new}.

In this section we show bounds on the MT distribution analogous to those shown by \cite{haeupler-saha-srinivasan} for the original MT algorithm. We also examine how the parametrization in terms of $\lambda$ gives particularly simple formulas, which can be useful even for analyzing the MT algorithm. We illustrate by using the MT distribution (for the original MT algorithm) for bounds on weighted independent transversals.

We need two preliminary definitions. For $E \in \mathcal A$ and $\mathcal B \subseteq \mathcal A$, define
$$
\mathcal B[E] = \{ B \in \mathcal B \mid E \not \subseteq B  \}
$$

We also define the \emph{strict neighbor-set} of $E$ as follows:
\begin{definition}[\textbf{Strict neighbor-set}]
For a set $E \in \mathcal A$ and a set $\mathcal T \subseteq  \mathcal A$, we say that $\mathcal T$ is a \emph{strict neighbor-set} for $E$ (or $\mathcal T \in \sns(E)$) if the following conditions hold:
\begin{enumerate}
\item Every $Z \in \mathcal T$ has $Z \sim E$.
\item There do not exist $Z, Z' \in \mathcal T$ with $Z \sim Z'$.
\end{enumerate}
\end{definition}

\begin{theorem}
\label{dist-thm1}
Let $E \in \mathcal A$ and suppose that $\mu$ satisfies Theorem~\ref{resample-main-thm} for the set of bad-events $\mathcal B[E]$. Then, when we run the PRA on events $\mathcal B$, we have
$$
\pmt(E) \leq p^E \sum_{\mathcal T \in \sns(E)} \prod_{Y \in \mathcal T} \mu(Y) \leq p^E \prod_{Y \sim E} (1 + \mu(Y))
$$
\end{theorem}
\begin{proof}
We assume that there is no bad-event $B \subseteq E$; for, if so, the the PRA output can never satisfy $E$ and so the result holds trivially.

We use a coupling construction following that of \cite{haeupler-saha-srinivasan}. Consider running the PRA with the set of bad-events $\mathcal B' = \mathcal B[E] \cup \{ E \}$ and fractional hitting-set $Q'$ defined by
$$
Q'(Y) = \begin{cases}
1 & \text{if $Y = E$} \\
Q(Y) & \text{otherwise}
\end{cases}
$$

When we run the PRA on $\mathcal B'$, we make a small change: whenever there is a choice of bad-event to resample, we will always choose to resample $E$ before $B \in \mathcal B$ if possible. We take advantage here of our freedom to select an arbitrary bad-event to resample, if there are multiple choices.   Observe that the PRA on $\mathcal B$ has identical behavior to the PRA on $\mathcal B'$, up to the first time $t$ when $E$ is true.  Since $Q(E) = 1$ and there are no bad-events $B \subseteq E$, the PRA on $\mathcal B'$ will have $E$ as its resampled set at time $t$. The witness tree $\hat \tau^t$ thus has its root node labeled $E$. 

We make a number of other observations about this tree $\hat \tau^t$. First, every subtree of $\hat \tau^t$ is a proper tree-structure with respect to the set of bad-events $\mathcal B[E]$. The reason for this is $E$ is never the resampled set before time $t$, and so does not affect the generation of $\hat \tau^t$.  Second, consider the set of children $w_1, \dots, w_s$ of the root node of $\hat \tau^t$. By Proposition~\ref{proper-ts}, these have distinct labels and $\{L(w_1), \dots, L(w_s) \} \in \ns( E )$. Suppose that $L(w_i) \bowtie E$. Then $E \subseteq B, L(w_i) \subseteq B$ for some bad-event $B \in \mathcal B'$. By definition of $\mathcal B'$ this is only possible if $E = B$ in which case $L(w_i) = E \not \bowtie E$. So in fact $\{ L(w_1), \dots, L(w_s) \} \in \sns(E)$.

Define $\Gamma^*$ to be the set of all tree-structures satisfying these properties. By Lemma~\ref{couple-lemma}, we have:
\begin{align*}
\pmt(E) \leq [\text{some $\tau \in \Gamma^*$ appears during execution of PRA on $\mathcal B'$}]  \leq \sum_{\tau \in \Gamma^*} w(\tau)
\end{align*}

Now, any $\tau \in \Gamma^*$ has a root node labeled $E$, and its children $w_1, \dots, w_s$ have distinct labels $\{ L(w_1), \dots, L(w_s) \} \in \sns(E)$. Furthermore, the subtrees of $w_1, \dots, w_s$ are proper tree-structures with respect to $\mathcal B[E]$. Thus, by Proposition~\ref{wt-bound-prop}, we have
\begin{align*}
\sum_{\tau \in \Gamma^*} w(\tau) &\leq p^E Q(E) \sum_{\mathcal T \in \sns(E)} \prod_{Y \in \mathcal T} \sum_{\tau_Y \in \Gamma(Y)} w(\tau_Y) \leq p^E Q(E) \sum_{\mathcal T \in \sns(E)} \prod_{Y \in \mathcal T} \mu(Y)
\end{align*}

Finally, note that $Q(E) = 1$, and we have shown the first bound on the probability of $E$.
\end{proof}

We can extend this result to the setting of Theorem~\ref{resample-main-thm-2} and Theorem~\ref{resample-main-thm2}:
\begin{corollary}
\label{dist-thm1-2}
Let $E \in \mathcal A$  and suppose that $\mu$ satisfies Theorem~\ref{resample-main-thm-2} for events $\mathcal B[E]$. Then
$$
\pmt(E) \leq p^E \sum_{\mathcal T \in \sns(E)} \prod_{Y \in \mathcal T} \sum_{k} \mu(Y,k) \leq p^E \prod_{Y \sim E} (1 + \sum_k \mu(Y,k))
$$
\end{corollary}
\begin{proof}
Consider forming some strict neighbor-set $\mathcal T \in \sns(E)$, with respect to the expanded set of bad-events $\tilde {\mathcal B}$ over the set of elements $\tilde {\mathcal X}$. If $\mathcal T$ is not good, then its contribution $\prod_{Y \in \mathcal T} \tilde \mu(Y)$ is zero. If $\mathcal T$ is good, then it can corresponds to $\{ (Y_1, k_1), \dots, (Y_r, k_r) \}$, where $\{ Y_1, \dots, Y_r \}$ is a strict neighbor-set of $E$ (with respect to the original set of elements $\mathcal X$), and furthermore  $\prod_{Y \in \mathcal T} \tilde \mu(Y) = \prod_{i=1}^r \mu(Y_i, k_i)$. Thus,
\[
\sum_{\substack{\mathcal T \subseteq \tilde {\mathcal A} \\ \mathcal T \in \sns(E)}}  \prod_{Y \in \mathcal T} \tilde \mu(Y) \leq \sum_{\substack{\mathcal T \subseteq \tilde {\mathcal A} \\ \mathcal T \in \sns(E)}} \prod_{Y \in \mathcal T} \sum_{k \in [K]} \mu(Y,k) \qedhere
\]
\end{proof}

\begin{theorem}
\label{dist-thm2}
If $\lambda$ satisfies Theorem~\ref{resample-main-thm2}, then any atomic set $E$ has $\pmt(E)  \leq \lambda^E$.
\end{theorem}
\begin{proof}
Let us enumerate $\mathcal T \in \sns(E)$ to apply Theorem~\ref{dist-thm1-2}. For each $(i,j) \in E$, the set $\mathcal T$ may contain one or zero sets $Y \sim i$. Therefore Theorem~\ref{dist-thm1-2} gives:
\begin{align*}
\pmt(E) &\leq p^E \sum_{\mathcal T \in \sns(E)} \prod_{Y \in \mathcal T} \sum_k \mu(Y,k) \leq p^E \prod_{(i,j) \in E} (1 + \sum_{Y \sim i} \sum_k \mu(Y,k)) \\
&\leq p^E \prod_{(i,j) \in E} \lambda_i =\lambda^E \qquad \text{(by Proposition~\ref{resample-main-thm3})} \qedhere
\end{align*}
\end{proof}

We can obtain a stronger bound than Theorem~\ref{dist-thm2} when $E$ is defined by a \emph{single variable}. In order to state this result, it is useful to define
$$
H_{i,j} = \sum_k \sum_{Y \ni (i,j)} \frac{Q_k(Y) \lambda^Y}{1 - S_k}
$$
and similarly the ``summation notation'' $H_i = \sum_j H_{i,j}$. We also assume throughout that $\lambda_{i,j} \geq H_{i,j}$ for each $(i, j)$ (if not, simply set  set $\lambda_{i,j} = 0$).

\begin{theorem}
\label{distthm3}
Let $u \in [n]$, and suppose that $\lambda$ satisfies Theorem~\ref{resample-main-thm2}. Let $J \subseteq F_u$, where recall $F_u$ is the set of possible assigned value to variable $X_u$. Then:
$$
\pmt(X_u \in J) \leq \frac{ \sum_{j \in J} \lambda_{u,j} }{\lambda_u - H_u + \sum_{j \in J} H_{u,j}}
$$
\end{theorem}
\begin{proof}
We define a function $\mu$ for the set of bad-events $\mathcal B[E]$; note that $\mathcal B[E]$ is derived from $\mathcal B$ by removing every bad-event $B \in \mathcal B$ such that $(i,j) \in B$ for some $j \in J$. We therefore define 
$$
\mu(Y,k) = 
\begin{cases}
\frac{\lambda^Y Q_k(Y)}{1 - S_k} & \text{if $u \not \sim Y$} \\
\alpha \frac{\lambda^Y Q_k(Y)}{1 - S_k}  & \text{if $(u,j) \in Y$ for $j \notin J$} \\
0  & \text{if $(u,j) \in Y$ for $j \in J$}
\end{cases}
$$
where $\alpha \in [0,1]$ is some parameter to be determined. We need to check that $\mu$ satisfies Theorem~\ref{resample-main-thm-2} with respect to $\mathcal B[E]$. This is nearly identical to the proof of Theorem~\ref{resample-main-thm2}; the only difficult case is to check the condition on $\mu(Y,k)$ where $(u,j) \in Y$ and $j \notin J$. For this, we have
{\allowdisplaybreaks
\begin{align*}
&p^Y Q_k(Y) \sum_{\mathcal T \in \gns(Y,k)} \prod_{ (Y', k') \in \mathcal T } \mu(Y', k') \\
& \qquad \leq p^Y Q_k(Y) (1 + \sum_{Z \bowtie_k Y} \mu(Z,k)) (1 + \sum_{Z \sim u} \sum_k \mu(Z,\ell)) \prod_{i \sim Y, i \neq u} (1 + \sum_{Z \sim i} \sum_{l} \mu(Z, \ell))  \\
& \qquad \leq p^Y Q_k(Y) \frac{1}{1 - S_k} (1 + \sum_{Z \sim u} \sum_k \mu(Z,\ell)) \prod_{i \sim Y, i \neq u} \lambda_i \qquad \text{(by Proposition~\ref{resample-main-thm3})} \\
& \qquad \leq \frac{p^Y Q_k(Y)}{1 - S_k} (1 +  \sum_\ell \sum_{j \notin J} \sum_{(u,j) \in Z} \frac{ \alpha \lambda^Z Q_\ell(Z)}{1 - S_\ell})  \prod_{i \sim Y, i \neq u} \lambda_i \\
& \qquad = \frac{p^Y Q_k(Y)}{1 - S_k} (1 + \alpha \sum_{j \notin J} H_{u,j} ) \prod_{i \sim Y, i \neq u} \lambda_i = \frac{\lambda^Y Q_k(Y)}{1 - S_k} \times  \frac{1 + \alpha \sum_{j \notin J} H_{u,j}}{\lambda_u}
\end{align*}
}

Since $\mu(Y,k) = \alpha \frac{\lambda^Y Q_k(Y)}{1 - S_k}$, it suffices to satisfy
$$
\alpha \frac{\lambda^Y Q_k(Y)}{1 - S_k} \geq \frac{\lambda^Y Q_k(Y)}{1 - S_k} \times  \frac{1 + \alpha \sum_{j \notin J} H_{u,j}}{\lambda_u} 
$$

So, after a little algebra, it suffices to take:
$$
\alpha = \frac{1}{\lambda_u - \sum_{j \notin J} H_{u,j}} = \frac{1}{\lambda_u - H_u + \sum_{j \in H} H_{u,j}}
$$
Note that $\lambda_u \geq 1 + \sum_j H_{u,j}$, so $\alpha$ is indeed in the range $[0,1]$ as desired. 

Corollary~\ref{dist-thm1-2} thus gives:
{\allowdisplaybreaks
\begin{align*}
\pmt(X_u \in J) &= \sum_{j \in J} p_{u,j} \sum_{\mathcal T \in \sns(E)} \prod_{Y \in \mathcal T} \sum_{k} \mu(Y,k) \\
&= \bigl( \sum_{j \in J} p_{u,j} \bigr) \bigl( 1 + \sum_{\ell \notin J} \sum_{Y \ni (u,\ell)} \sum_k \alpha \frac{\lambda^Y Q_k(Y)}{1 - S_k} \bigr) \\
&= \bigl( \sum_{j \in J} p_{u,j} \bigr) \bigl( 1 + \alpha \sum_{\ell \notin J} H_{u,l\ell} ) = \bigl( \sum_{j \in J} p_{u,j} \bigr) \bigl( 1 + \frac{\sum_{\ell \notin J} H_{u,\ell} }{\lambda_u - \sum_{\ell \notin J} H_{u,\ell}} \bigr) \\
&= \bigl( \sum_{j \in J} p_{u,j} \bigr) \bigl( \frac{\lambda_u}{\lambda_u - \sum_{\ell \notin J} H_{u,\ell}} \bigr) = \frac{ \sum_{j \in J} \lambda_{u,j} }{\lambda_u - H_u + \sum_{j \in J} H_{u,j}} \qedhere
\end{align*}
}
\end{proof}

\begin{corollary}
\label{distcorr2}
Suppose that $\lambda$ satisfies Theorem~\ref{resample-main-thm2} and let $J \subseteq F_u$.  Then:
$$
\pmt(X_u \in J) \geq \frac{\sum_{j \in J} \lambda_{u,j} - \sum_{j \in J} H_{u,j}}{\lambda_u - \sum_{j \in J} H_{u,j}}
$$
\end{corollary}
\begin{proof}
Apply Theorem~\ref{distthm3} to bound from above the probability of $X_u \in F_u - J$.
\end{proof}

\subsection{The MT distribution and independent transversals}
To illustrate our MT distribution results, let us  consider \emph{weighted} transversals, as discussed in \cite{aharoni-berger-ziv}. Suppose $G$ has maximum degree $\Delta$ and each block of $G$ has size exactly $b$. Given a weighting function $w: V \rightarrow [0, \infty)$, we may wish to find an independent transversal of minimum or maximum weight. Clearly, $G$ has a transversal (not necessarily independent) of weight at most (respectively at least) $w(V)/b$; we would like to find an independent transversal whose weight is comparable to this.

One effective method to find weighted independent transversals is to find a \emph{strong coloring of $G$}, which is a decomposition $V = I_1 \sqcup \dots \sqcup I_b$,  wherein each $I_i$ is an independent transversal of $G$. Clearly, given such  a strong coloring of $G$, we can find in polynomial time an independent transversal $I$ such that $w(I) \geq w(V)/b$ (respectively, $w(I) \leq w(V)/b$).  When $b$ is large compared to $\Delta$, then such strong colorings exists and can even be found efficiently.
 \begin{proposition}[\cite{haxell2008improved}, \cite{harris2014constructive}, \cite{graf2018finding}]
   \label{5cor}
When $b \geq \tfrac{11}{4} \Delta$ and $\Delta \geq \Delta_0$ for some constant $\Delta_0$, then a strong coloring of $G$ exists. When $b \geq 5 \Delta$, or when $b \geq 3 \Delta +1$ and $\Delta$ is constant, then a strong coloring of $G$ can be found in randomized polynomial time.
\end{proposition}

A more general method to find such weighted independent transversals is given by \cite{aharoni-berger-ziv} via \emph{fractional} strong colorings. A fractional strong coloring of $G$ is a probability distribution $\Omega$ over independent transversals $I$, with the property that any vertex $v$ has  $P_{\Omega}(v \in I) = 1/b$.

\begin{proposition}[\cite{aharoni-berger-ziv}]
When $b \geq 2 \Delta$, there exists a fractional strong coloring of $G$. In particular, there exists an independent transversal $I$ with $w(I) \geq w(V)/b$ (respectively, $w(I) \leq w(V)/b$).
\end{proposition}

When $b \geq 4 \Delta$, we can use the MT algorithm to construct independent transversals whose weight can be upper-bounded or lower-bounded in terms of $w(V)$. (When $b \geq 5 \Delta$, then these follow immediately from Proposition~\ref{5cor} already.) The two constructions are very similar, so we summarize them here together. We apply the PRA, in which there is a variable $X_i$ corresponding to each block $V_i$; we set $X_i = v$ to mean that $v \in V_i \cap I$. We will apply Theorem~\ref{resample-main-thm2} by selecting a subset $B_i \subseteq V_i$ of size $|B_i| = r \geq 4 \Delta$, and defining
$$
\lambda_{i,v} = \alpha [v \in B_i] \qquad \text{ where } \alpha = \frac{r - \sqrt{r} \sqrt{r-4 \Delta}}{2 r \Delta}
$$

We take as our set of bad-events a single family $\mathcal B_1$, which contains a separate atomic bad-event corresponding to each edge. We also use the trivial hitting-set. With this choice, $\bowtie_1$ is null. So $\lambda_i = r \alpha$ and we have
$$
G_i(Q, \lambda) = \sum_{v \in B_i} \sum_{\substack{\text{edges $f$}\\ \text{involving $v$}}} \alpha^2 \leq r \Delta \alpha^2
$$

Thus, Theorem~\ref{resample-main-thm2} is satisfied. Furthermore, for any block $i$ and $v \in B_i$ we have $H_{i,v} \leq \alpha^2 \Delta$. We will show that, by selecting the sets $B_i$ appropriately, the expected weight of the resulting independent transversal satisfies certain bounds. We can easily achieve an actual independent transversal whose weight is close to the expected weight by a polynomial number of repetitions.

\begin{theorem}
\label{thm:weighted-transversals}
Let $G$ be a graph of maximum degree $\Delta$ whose vertex set is partitioned into blocks of size exactly $b$. If $4 \Delta \leq b \leq 4.5 \Delta$, then there is an efficient procedure to randomly sample an independent transversal $I$ of $G$ such that
$$
\bE[w(I)] \geq w(V) \Bigl( \frac{\sqrt{b} + \sqrt{b - 4 \Delta}}{\sqrt{b} (2 b - 1) + \sqrt{b - 4 \Delta}} \Bigr) \geq \frac{w(V)}{8 \Delta - 1}.
$$

If $4.5 \leq b \leq 5 \Delta$, then there is an efficient procedure to randomly sample an independent transversal $I$ of $G$ such that
$$
\bE[w(I)] \geq \frac{4  w(V)}{27 \Delta}
$$
\end{theorem}
\begin{proof}
In the first result, we set $B_i = V_i$, of size $r = b$. Corollary~\ref{distcorr2} gives
$$
\pmt(v \in I) \geq \frac{\lambda_{i,j} - H_{i,v}}{\lambda_i - H_{i,v}} \geq \frac{\alpha - \alpha^2 \Delta}{b \alpha - \alpha^2 \Delta} = \frac{\sqrt{b} + \sqrt{b - 4 \Delta}}{\sqrt{b} (2 b - 1) + \sqrt{b - 4 \Delta}}
$$

For the second result, sort the vertices in decreasing order of weight within block $i$ as $v_{i,1}, v_{i,2}, \dots, v_{i,b}$ where $w(v_{i,1}) \geq w(v_{i,2}) \geq \dots \geq w(v_{i,b})$. We take $B_i = \{v_{i,1}, \dots, v_{i,r} \}$ where $r = \lceil 9 \Delta/2 \rceil$. By Corollary~\ref{distcorr2}, for any block $i$ we have
\begin{align*}
\pmt(X_i = v_{i,j}) \geq \frac{ \alpha - \alpha^2 \Delta }{r \alpha - \alpha^2 \Delta} \geq \frac{ \alpha - \alpha^2 \Delta }{(9 \Delta/2 + 1/2) \alpha - \alpha^2 \Delta}
\end{align*}

Routine algebraic calculations show that this expression is lower-bounded by $q = \frac{4}{27 \Delta}$.  Now consider some block $V_i$, and write $w_j = w(v_{i,j})$ for $j \leq b$. As $X_i \in B_i$ with probability one, we have
$$
\bE[w(V_i \cap I)] =  w_{r} + \sum_{j=1}^{r} \Pr(X_i = v_{i,j}) (w_j - w_{r}) \geq w_r + \sum_{j=1}^r q (w_j - w_r)
$$

Define $t = \sum_{j=1}^{r} w_j$. Since the vertices are in sorted order, $w_{r} \geq \frac{w(V_i) - t}{b - r}$ and so
\begin{align*}
\bE[ w(V_i \cap I) ] &\geq w_{r} +  q (t - r w_{r})=  w_{r} ( 1 - r q) + q t  \\
&\geq \frac{w(V_i) - t}{b-r} (1 - r q) + q t = -t  \frac{ (1-b q)}{b - r} + \frac{w(V_i) (1 - r q)}{b-r} \\
&\geq -w(V_i)  \frac{ (1-b q)}{b - r} + \frac{w(V_i) (1 - r q)}{b-r} = q w(V_i) \qquad \text{since $t \leq w(V_i)$ and $b q  \leq 1$}
\end{align*}

The result follows by linearity of expectation, summing over all blocks $i$.
\end{proof}

\begin{theorem}
\label{thm:weighted-transversals2}
Let $G$ be a graph of maximum degree $\Delta$ whose vertex set is partitioned into blocks of size at least $b$. If $b \geq 4 \Delta$, there is an efficient procedure to randomly sample an independent transversal $I$ of $G$ such that
$$
\bE[w(I)] \leq w(V) \frac{2}{b + 4 \sqrt{(b - 4 \Delta) \Delta}}
$$
\end{theorem}
\begin{proof}
We may assume $b \leq 5 \Delta$, as otherwise this follows from Proposition~\ref{5cor}.

Sort the vertices in increasing order of weight, so that in each block $i$ we have $w(v_{i,1}) \leq w(v_{i,2}) \leq \dots \leq w(v_{i,b})$. As before, let us write $X_i = j$ as shorthand for $X_i = v_{i,j}$. In this case, we apply our construction with $r = 4 \Delta$ and $B_i = \{v_{i,1}, \dots, v_{i,r} \}$

Let us fix a block $i$, and write $w_j = w(v_{i,j})$ for $j \leq b$. Then
\begin{equation}
\label{wve1}
w(V_i \cap I) =  w_1 + (w_2-w_1) [ X_i \geq 2] + (w_3-w_2) [X_i \geq 3] + \dots + (w_r - w_{r-1}) [X_i = r]
\end{equation}

We now have $H_{i,v} \leq \alpha^2 \Delta = \frac{1}{4 \Delta}$. So for any $j \geq 1$, Corollary~\ref{distthm3} gives $\Pr(X_i \geq j) \leq \frac{ (r - j + 1) \alpha}{r \alpha - (j-1) \frac{1}{4 \Delta}} = \frac{2 (r - j + 1)}{2r - j + 1} =:q_j$. With this notation, taking the expectation of (\ref{wve1}) gives
$$
\bE[w(V_i \cap I)] \leq w_1 q_1 + (w_2 - w_1) q_2 + (w_3 - w_2) q_3 + \dots + (w_r - w_{r-1}) q_r
$$

Let us define $u_j = w_j - w_{j-1} \geq 0$ for $j \geq 2$, and $u_1 = w_1$. Noting that $\sum_j w_j = \sum_j u_j (b - j + 1)$, we can write this as:
\begin{align*}
\bE[w(V_i \cap I)] &\leq \sum_{j=1}^r q_j u_j  = \sum_{j=1}^r u_j (b-j+1) \times \frac{2 (r-j+1)}{(2 r - j + 1) ( b-j+1 )} \\
&\leq  w(V_i) \max_{j \in [r]} \frac{2 (r-j+1)}{(2 r - j + 1) (b-j+1) } \leq w(V_i) \max_{x \in [0,r-1]} \frac{2 (r - x) }{(b - x)(2 r - x)}
\end{align*}

We can view this expression $f(x) = \frac{2 (r - x) }{(b - x)(2 r - x)}$ as a differentiable function of $x$, which has critical points at $x = r \pm \sqrt{ br - r^2}$. As $b \leq 8 \Delta$, the function $f(x)$ achieves its maximum value at $x = r - \sqrt{b r - r^2}$, and we have there
$$
f(x) = \frac{2}{b + 2 \sqrt{(b-r) r}} = \frac{2}{b + 4 \sqrt{(b - 4 \Delta) \Delta}}
$$

So $\bE[w(V_i \cap I)] \leq \frac{2 w(V_i)}{b + 4 \sqrt{(b - 4 \Delta) \Delta}}$. The result follows by linearity of expectation, summing over all blocks $i$.
\end{proof}

\section{Algorithmically implementing the hitting set for Theorem~\ref{devthm}}
\label{access-q-sec}
To implement the PRA using the fractional hitting-set $Q$ of Theorem~\ref{devthm}, we must efficiently sample from $Q$, in the following sense: given an atomic  bad-event $B$  on elements $x_1, \dots, x_k$ with weights $a_1, \dots, a_k$, we  must select a subset $Y \subseteq B$ with probability proportional to $Q(Y)$. Note that implemented naively this step might take $\binom{n}{d}$ time, which is potentially exponential.

For any set $W \subseteq B$, define
$$
R(W) = \sum_{\substack{Y: W \subseteq Y \subseteq B\\|Y| = d}} Q(Y)
$$
This can be evaluated in time $O(d k)$ using a dynamic program. To efficiently sample $Y \subseteq B$ with probability proportional to $Q(Y)$,  use the following procedure:

\begin{algorithmic}[1]
\STATE Let $Y_0 = \emptyset$
\FOR{$i = 1, \dots, d$}
\STATE \textbf{for each} $j \in B - Y_{i-1}$ \textbf{do} compute $q_j = R( Y_{i-1} \cup \{ j \})$.
\STATE Select some $j \in B - Y_{i-1}$ with probability proportional to $q_j$.
\STATE Set $Y_i = Y_{i-1} \cup \{ j \}$.
\ENDFOR
\RETURN $Y_d$
\end{algorithmic}

\begin{proposition}
For any $Z \subseteq B$ with $|Z| = d$, we have $\Pr(Y_d = Z) =\frac{Q(Z)}{\sum_{Y \subseteq B} Q(Y)}$.
\end{proposition}
\begin{proof}
We show by induction on $i$ the following: for any sets $W \subseteq Z \subseteq B$ with $|W| = i, |Z| = d$, and $0 \leq i \leq d$,  we have
$$
\Pr(Y_d = Z \mid Y_i = W) = \frac{Q(Z)}{R(W)}
$$

Applying this with $i = 0, W = \emptyset$ will give us the desired result. Also, the induction case with $i = d$ is trivially true. For the induction step:
{\allowdisplaybreaks
\begin{align*}
\Pr(Y_d = Z \mid Y_i = W) &= \frac{ \sum_{z \in Z - W} R( W \cup \{z \} ) \Pr( Y_d = Z \mid Y_{i+1} = W + z)}{\sum_{x \in B - W} R(W \cup \{x \})} \\
&= \frac{ \sum_{z \in Z - W}  R( W \cup \{z \}) Q(Z) / R(W + z) } {\sum_{x \in B - W} R(W \cup \{x \})} \qquad \text{induction hypothesis} \\
&= \frac{ Q(Z) (d - i)  } {\sum_{x \in B - W} \sum_{Y: W \cup \{x \} \subseteq Y \subseteq Z} Q(Y)} \\
&= \frac{ Q(Z) (d - i)  } {\sum_{Y: W \subseteq Y \subseteq Z} \sum_{x \in Y - W} Q(Y)} = \frac{ Q(Z) (d - i)  } {R(W) (d-i)} = \frac{Q(Z)}{R(W)}
\end{align*}
}
thus completing the induction.
\end{proof}

\section{Functional analysis for Theorem~\ref{csp-thm}}
\label{csp-proof-app}
We prove that the vector $b_k$ given in Theorem~\ref{csp-thm} satisfies the stated properties (C1), (C2).
  
\begin{proposition}
\label{csp-prop2}
Let $c \geq 1$, let $D \geq 2$, and let $\epsilon$ satisfy $0 < \epsilon \leq 1/D$. Then the quantity $b$ given below satisfies properties (C1), (C2):
$$
b = \begin{cases}
\frac{100 \ln(1/\epsilon)}{1 + \ln( \frac{\ln(1/\epsilon)}{c} )} & \text{for $c \leq \ln(1/\epsilon)$} \\
c(1+\epsilon) + 10 \sqrt{c \ln \bigl( D + \frac{1}{c \epsilon^2}  \bigr)} & \text{for $c > \ln(1/\epsilon)$}
\end{cases}
$$
\end{proposition}
\begin{proof}
Let $\delta = \ln(1/\epsilon)$ and let $d = \ln D$. As $D \geq 2$, we have $\epsilon \leq 1/2$ and hence $\delta \geq \ln 2$. We also write $\mu = c(1+\epsilon)$. It is immediately clear that $b \geq \mu$.

\textbf{Case I: $\pmb{c \leq \delta}$.} Let $x = \delta/c \geq 1$, so $b = \frac{100 \delta}{1 + \ln x}$. We estimate $\text{Ch}(\mu, b)$ as:
\begin{align*}
\text{Ch}(\mu,b) &= e^{b - c(1+\epsilon)} \Bigl( \frac{c(1+\epsilon)}{b} \Bigr)^b \leq e^{b} \Bigl( \frac{2 c}{b} \Bigr)^b = (\frac{2 e (\delta/x) (1 + \ln x)}{100 \delta})^{100 \delta / (1 + \ln x)} = \Bigl( \frac{e (1 + \ln x)}{50 x} \Bigr)^{ \frac{100 \delta}{1 + \ln x}}
\end{align*}

Simple calculus shows that, for $x \geq 1$, we have $\Bigl( \frac{e (1 + \ln x)}{50 x} \Bigr)^{1/ (1 + \ln x)} \leq e^{e/50 - 1} \leq 0.389$. So 
$$
\text{Ch}(\mu,b) \leq (0.389)^{100 \delta} \leq (0.389)^{100 \ln 2} \leq 3.78 \times 10^{-29} \leq 1/2.
$$
So (C2) is satisfied. To show (C1):
\begin{align*}
\Bigl( \frac{b + 1}{c (1 + \epsilon)} - 1 \Bigr) \text{Ch}(\mu, b) \times 4 D / \epsilon \leq \Bigl( \frac{2 b}{1}  \Bigr) \text{Ch}(\mu, b) \times 4 \epsilon^{-2}  \leq 800 \delta \text{Ch}(\mu, b) \times \epsilon^{-2} \leq \delta (0.389)^{100 \delta} \times e^{2 \delta}
\end{align*}

Simple calculus shows that, for $\delta \geq \ln 2$, this attains a maximum value of $1.048 \times 10^{-28}$, which is attained at $\delta = \ln 2$. In particular, it is smaller than $1$.

\textbf{Case II: $\pmb{c > \delta}$.} Let $v = D + \epsilon^{-2}/c$; clearly $D \leq v$, and as $\epsilon \leq 1/D$ we also have the crude bound $v \leq \epsilon^{-1} + \epsilon^{-2}/ \delta \leq \epsilon^{-3}$. With this notation, we have $b = \mu + 10 \sqrt{c \ln v}$. The relative deviation between $\mu$ and $b$ here is given by $\lambda = b/\mu - 1 =   10 \frac{\sqrt{c \ln v}}{c (1+\epsilon)}$. We observe the following bound on the size of $\lambda$:
$$
\lambda = \frac{10 \sqrt{\ln v}}{\sqrt{c} ( 1 + \epsilon)} \leq \frac{10 \sqrt{\ln (\epsilon^{-3})}}{\sqrt{\delta}} = 10 \sqrt{3} \leq 17.4
$$

Since $\lambda \leq 17.4$, a simple calculation shows that $\text{Ch}(\mu, \mu(1 + \lambda)) \leq e^{-\mu \lambda^2/10}$; thus, $\text{Ch}(\mu, b) \leq v^{-\frac{10}{1+\epsilon}} \leq v^{-6.6}$. As $v \geq D \geq 2$, this is at most $0.0104$, and thus (C1) is satisfied. To show (C2):
{\allowdisplaybreaks
\begin{align*}
\Bigl( \frac{b + 1}{c (1 + \epsilon)} - 1 \Bigr) \text{Ch}(c (1 + \epsilon),b)  
&\leq \Bigl( \frac{c (1+\epsilon) + 10 \sqrt{c \ln v} + 1 - c(1+\epsilon)}{c (1 + \epsilon)} \Bigr) \times v^{-6.6} \\
&= \Bigl( \frac{10 \sqrt{\ln v} + 1/\sqrt{c}}{\sqrt{c} (1 + \epsilon)} \Bigr) \times v^{-6.6} \\
&\leq 5 v^{-5.6} c^{-{1/2}} \qquad \text{as $10 \sqrt{\ln v} + 1 \leq 5 v$ for $v \geq D \geq 2$} \\
&= 5 (D + \epsilon^{-2}/c)^{-5.6} c^{-1/2}
\end{align*}
}
Simple analysis shows that this quantity $5 (D + \epsilon^{-2}/c)^{-5.6} c^{-1/2}$, as a function of $c$, has a maximum value at $c = 10.2/(D \epsilon^2)$, at which point we have
$$
5 (D + \epsilon^{-2}/c)^{-5.6} c^{-1/2} \leq \frac{0.927277 \epsilon}{D^{5.1}}
$$

Therefore, as $D \geq 2$, we have shown that
\begin{align*}
  \Bigl( \frac{b + 1}{c (1 + \epsilon)} - 1 \Bigr) \text{Ch}(c (1 + \epsilon),b)   &\leq \frac{0.927277 \epsilon}{D^{5.1}} \leq \frac{ \epsilon}{4 D}
\end{align*}
satisfying (C2).
\end{proof}

\section{Full proof of Theorem~\ref{final-thm}}
\label{full-proof-app}
In this section, we extend the proof of Theorem~\ref{final-thm} to cover the case when $C+D \leq 2^{896}$. 

\begin{proposition}
\label{noncons1a}
Suppose the original problem instance has congestion $C$ and dilation $D$. Let $i', C', d$ be positive integers with $d \leq C'$, and let $\beta > 1$. Define 
$$
p = \frac{(i' \beta)^d} {d! \binom{C'+1}{d}}
$$

If $p < 1$ and  $\beta -  \frac{D d p}{1-p} \geq 1$, then there is a schedule $S$ of length $L \leq C+D$, in which every interval of length $i'$ has congestion at most $C'$. Furthermore, $S$ can be found in polynomial time.
\end{proposition}
\begin{proof}
We add delays in the range $b = \{0, \dots, C-1 \}$ uniformly to each packet. In this case, we have a variable corresponding to each packet $x$, and for each delay $t$ we assign $\lambda_{x,t} = \beta/C$. For each edge $f$ and $i'$-interval $I$, we have a complex bad event $\mathcal B_{f, I}$ that the congestion in the interval exceeds $C'$. We use the fractional hitting-set $Q_{f,I}$ with parameter $d$ as described in Theorem~\ref{devthm}.

For a given $f, I$,  we first must compute $\mu_{f,I}$, which is the sum of $\lambda_{x,t}$ over all packets $x$ and delays $t$ contributing to congestion for $f, I$. There are at most $C$ packets which pass through $f$, and at most $i'$ choices for $t$ cause a transit of $f$ within $I$. So $\mu_{f,I} \leq C i' \beta/C = i' \beta$. The bad event is that this at least $C' + 1$. By Theorem~\ref{devthm}, this gives $S_{f,I} = S(\mathcal B_{f,I}, Q_{f,I}, \lambda) \leq \frac{\mu_{f,I}^d} {d! \binom{C' + 1}{d}} \leq p.$

Let $\mu_{x,f,I}$ be the sum of $\lambda_{x,t}$ over all delays $t$ contributing to congestion on edge $f$ in interval $I$. By Theorem~\ref{devthm} we have
$$
\sum_{f,I} G_x(Q_{f,I}, \lambda) \leq \sum_{f,I} \frac{\mu_{x,f,I}}{\mu_{f,I}} \times \frac{d \mu_{f,I}^d}{d! \binom{C+1}{d}} \leq \sum_{f,I} \frac{d p \mu_{x,f,I}}{i' \beta} 
$$

For each of the $D$ edges traversed by $x$, a choice of delay $t$ contributes to congestion within $i'$ separate intervals. So $\sum_{f,I} \mu_{x,f,I} \leq C \times \beta/C \times D \times i' = D i' \beta$. So
$$
\sum_{f,I} \frac{G_x(Q_{f,I}, \lambda )}{1 - S_{f,I}} \leq \frac{p \times D i' \beta}{i' \beta (1 -p)} = \frac{D d p}{1-p}
$$

In order to apply Theorem~\ref{resample-main-thm2}, each packet $x$ must satisfy the constraint
\begin{equation}
\label{hh2}
\lambda_x \geq  1 + \sum_{f,I} \frac{G_x(Q_{f,I}, \lambda)}{1 - S_{f,I}}.
\end{equation}

We have $\lambda_x = C \beta/C  = \beta$ and so (\ref{hh2}) becomes $\beta \geq 1 + \frac{D d p}{1-p}$, which is precisely the constraint specified in the hypothesis.
\end{proof}

\begin{proposition}
\label{noncons2a}
Suppose the original problem instance has congestion $C$ and dilation $D$ with $C + D \leq X = 100,000$. Then there is a schedule of length $L \leq C+D$, which satisfies the 4IC with respect to $T = 6, T' = 5$ which can be produced in polynomial time.
\end{proposition}
\begin{proof}
For each packet $x$ we add a random delay in the interval $\{0, \dots, C-1 \}$. For each edge $f$, and each 4-interval $I$ starting at time $s$ for odd integer $s$, we have a complex bad-event $\mathcal B_{f,I}$ that $$
\text{OF}^+ (c_s(f), c_{s+1}(f); T) + \text{OF}^{-}(c_{s+2}(f), c_{s+3}(f); T) > T'
$$

The overall proof here is very similar to Proposition~\ref{noncons2}: we will apply the PRA using separate variable for each packet (the value of a variable is the chosen delay), and the vector $\vec \lambda_{x,t} = \alpha = \beta/C$, where $\beta = 1.430599$.  We define $\Phi_0, \Phi_1, \Phi_2, \Phi_3, \Phi$ as in Proposition~\ref{noncons2}. 

Thus again $S_{f,I}  \leq \sum_{Y} Q_{f,I}(Y) \lambda^Y \leq \Phi \leq \hat \Phi$.  For any packet $x$ and delay $t$ and $j = 0, \dots, 3$, there are at most $D/2$ pairs $f,I$ for which packet $x,t$ has type $j$. Each such $f,I$ has $G_{x,t}(Q_{f,I}, \lambda) \leq \Phi_j$. Summing over all $C$ values of $t$ and all such $f, I$ gives
$$
\sum_{f, I} G_x (Q_{f,I}, \lambda) \leq \frac{C D}{2} (\Phi_0 + \Phi_1 + \Phi_2 + \Phi_3) \leq \frac{C X}{2} (\hat \Phi_0 + \hat \Phi_1 + \hat \Phi_2 + \hat \Phi_3)
$$

In order to apply Theorem~\ref{resample-main-thm2}, each packet $x$ must satisfy the constraint
$$
\lambda_x \geq 1 + \sum_{f,I} \frac{G_x(Q_{f,I}, \lambda)}{1 - S_{f,I}}
$$
We have $\lambda_x = \beta$, and so it suffices to satisfy the condition
\begin{equation}
\label{ss1alt}
\beta  - X/2 \times \frac{C (\hat \Phi_0 + \hat \Phi_1 + \hat \Phi_2 + \hat \Phi_3)}{1 - \hat \Phi} \geq 1
\end{equation}

We now observe that the LHS of (\ref{ss1alt}) is a decreasing function of $C$. To see this, note that for instance we have 
\begin{equation}
\label{r1alt2}
C \hat \Phi_0 = C \sum_{y_0, y_1, y_2, y_3} \tbinom{C - 1}{y_0 - 1} \tbinom{C}{y_1} \tbinom{C}{y_2} \tbinom{C}{y_3} b(y_0 , y_1, y_2, y_3) (\beta/C)^{y_0 + y_1 + y_2 + y_3 }
\end{equation}
which is an increasing function of $C$. A similar argument holds for $C \Phi_1, C \Phi_2, C \Phi_3, \Phi$. 

Because of this fact, it suffices to satisfy (\ref{ss1alt}) at $C = X$. We use a computer code similar to Proposition~\ref{noncons2} to select $b$ in this case. Overall, we get a bound
$$
\frac{\hat \Phi_0 + \hat \Phi_1 + \hat \Phi_2 + \hat \Phi_3}{1-\hat \Phi} \leq 8.16 \times 10^{-11}
$$
which satisfies (\ref{ss1alt}).
\end{proof}

\begin{theorem}
Suppose $2 \leq C + D < 2^{896}$. Then there is a feasible schedule of makespan at most $6.73 (C+D)$.
\end{theorem}
\begin{proof}
\textbf{Case I: $\pmb{2^{32} \leq C+D < 2^{896}}$.}

Apply Proposition~\ref{noncons1a} with $\beta = 1.00563, C' = 17040600, i' = 2^{24}, d = 250000$ to produce a schedule $S_1$ of length $L_1 \leq C+D$, 
in which each interval of length $2^{24}$ has congestion at most $17040600$.

Apply Proposition~\ref{noncons1} to $S_1$ with $\alpha = 5.98328\times 10^{-8}, C' = 1320, d = 270, m = 400$ to get a schedule $S_2$ of length $L_2 \leq 1.0025 L_1 + 2^{24}$, in which each interval of length $1024$ has congestion at most $1320$.

Apply Proposition~\ref{noncons2} to $S_2$ to get a schedule $S_3$ of length $L_3 \leq 1.02779 L_2 + 1024$ satisfying the 4IC with $T = 6, T' = 5$. By Proposition~\ref{4frameprop}, this yields a feasible schedule $S_4$ whose makespan is $L_4 \leq 6.5 L_3 + 4.5$. As $C+D \geq 2^{32}$, we have $L_4 \leq 6.73 (C+D)$.

\textbf{Case II: $\pmb{150,000 \leq C+D < 2^{32}}$.}
Apply Proposition~\ref{noncons1a} with $\beta = 1.01831, C' = 1320, i' = 1024, d = 280$ to produce a schedule $S_1$ of length $L_1 = C+D$, in which each interval of length $1024$ has congestion at most $1320$.

Apply Proposition~\ref{noncons2} to $S_1$ to get a schedule $S_2$ of length $L_2 \leq 1.02779 L_1 + 1024$ satisfying the 4IC with $T = 6, T' = 5$. By Proposition~\ref{4frameprop}, this yields a feasible schedule $S_3$ whose makespan is $L_3 \leq 6.5 L_2 + 4.5$. As $C+D \geq 2^{17}$, we have  $L_3 \leq 6.73 (C+D)$.

\textbf{Case III: $\pmb{2^{15} \leq C+D < 150,000}$.} Apply Proposition~\ref{noncons1a} with $\beta = 1.00202, C' = 675, i' = 512, d = 140$ to produce a schedule $S_1$ of length $L_1 = C+D$, in which each interval of length $512$ has congestion at most $675$.

We now apply a slight variant of Proposition~\ref{noncons2} to $S_1$, using parameters $m = 51, C = 675, i = 512, T = 6, T' = 5, \alpha = 0.002113$ to produce a schedule of length $L_2 = (1 + 1/51) L_1 + 512$, which satisfies the 4IC for $T = 6, T' = 5$.   By Proposition~\ref{4frameprop}, this yields a feasible schedule $S_3$ whose makespan is $L_3 \leq 6.5 L_2 + 4.5$.  As $C+D \geq 2^{15}$, we have  $L_3 \leq 6.73 (C+D)$.

\textbf{Case IV: $\pmb{20 \leq C+D \leq 2^{15}}$.} Apply Proposition~\ref{noncons2a} to produce a schedule $S_1$ of length $L_1 \leq C+D$ satisfying the 4IC for $T = 6, T' = 5$. By Proposition~\ref{4frameprop}, this yields a feasible schedule $S_2$ whose makespan is $L_2 \leq 6.5 L_1 + 4.5$.  As $C+D \geq 20$, this is at most $6.73 (C+D)$.

\textbf{Case V: $\pmb{2 \leq C+D \leq 19}$.} Apply Proposition~\ref{noncons1a} with $\beta = 1.27877, C' = 9, i' = 2, d = 8$ to produce a schedule $S_1$ of length $L_1 \leq C+D$, in which each interval of length $2$ has congestion at most $9$. By Proposition~\ref{peisprop}, this yields a feasible schedule of length $\lceil L_1/2 \rceil (9 + 1) \leq 5(C+D) + 5/2 \leq 6.25 (C+D)$.
\end{proof}
\end{document}